\documentclass{amsart}
\usepackage{amsfonts,amssymb,latexsym}
\usepackage{amsmath,amsthm}
\usepackage{eucal}
\usepackage[latin1]{inputenc}
\usepackage{color}
\definecolor{darkgreen}{rgb}{0.05, 0.5, 0.06}
\usepackage{a4wide}

\newtheorem{theorem}{Theorem}[section]
\newtheorem{lemma}[theorem]{Lemma}
\newtheorem{proposition}[theorem]{Proposition}
\newtheorem{corollary}[theorem]{Corollary}

\theoremstyle{definition}
\newtheorem{definition}[theorem]{Definition}

\newtheorem*{merci}{Acknowledgements}

\newtheorem{remark}{Remark}[section]
\newcommand{\R}{\mathbb{R}}
\newcommand{\C}{\mathbb{C}}

\newcommand{\Y}{\mathcal{Y}}
\newcommand{\X}{\mathcal{X}}

\newcommand{\dx}{\partial_x}

\newcommand{\ind}{{1\! \! 1}} %indicator function notation

\numberwithin{equation}{section}

\date{\today}

\begin{document}

\title[Dysthe equation ]{Global well-posedness and scattering for\\ 
the Dysthe equation in $L^2(\mathbb R^2)$}


\author[R. Mosincat]{Razvan Mosincat}
\address{Department of Mathematics, University of Bergen, Postbox 7800, 5020 Bergen, Norway}
\email{Razvan.Mosincat@uib.no}

\author[D. Pilod]{Didier Pilod}
\address{Department of Mathematics, University of Bergen, Postbox 7800, 5020 Bergen, Norway}
\email{Didier.Pilod@iuib.no}

\author[J.-C. Saut]{Jean-Claude Saut}
\address{Laboratoire de Math\' ematiques, UMR 8628\\ Universit\' e Paris-Saclay and CNRS \\91405 Orsay, France}
\email{jean-claude.saut@niversite-paris-saclay.fr}

\maketitle

\begin{center} 
\emph{Dedicated to Kristian B. Dysthe, with friendship and admiration}.
\end{center}

\vspace{0.5cm}

\begin{abstract} 
This paper focuses on the Dysthe equation which is a higher order approximation of the water waves system in the modulation (Schr\"{o}dinger) regime and in the infinite depth case. We first review the derivation of the Dysthe and related equations. Then we study the initial-value problem. We prove a small data global well-posedness  and scattering result in the critical space $L^2(\R^2)$. 
%This result is sharp in the sense that the flow map cannot be continuous below $L^2(\R^2)$.  
This result is sharp in view of the fact that the flow map cannot be $C^3$ continuous below $L^2(\mathbb R^2)$. 
Our analysis relies on linear and bilinear Strichartz estimates in the context of the Fourier restriction norm method. Moreover, since we are at a critical level, we need  to work in the framework of the atomic space $U^2_S $ and its dual  $V^2_S $ of square bounded variation functions. 
We also prove that the initial-value problem is locally well-posed in $H^s(\mathbb R^2)$, $s>0$. 
Our results extend to the finite depth version of the Dysthe equation.
\end{abstract}


\vspace{0.5cm}

\section{Introduction}

\subsection{Introduction of the model and physical motivation}

It is well-known that the full water waves system is too complex to be used to describe rigorously long time dynamics of waves except in physically trivial situations.  A well-known procedure (going back to Lagrange!) is instead to derive simpler asymptotic models in various relevant regimes. We refer to the book \cite{La} for a systematic, up to date and   rigorous approach. The first step is of course to find suitable non-dimensional small parameters.  This will allow to derive  approximations of the water waves system by expansions with respect to the parameters.  For instance, in order to derive rigorous error estimates, one expands the nonlocal Dirichlet to Neumann operator which appears in the Zakharov-Craig-Sulem formulation of the water waves problem, see \cite{La}.

In this  introductory section we motivate and recall the derivation of higher order 
nonlinear Schr\"{o}dinger-type equations in the context of surface water waves.

We are interested here in the so-called modulation regime which aims to describe the weakly nonlinear modulation of a train of surface gravity (or gravity-capillary) waves.  In other words, one investigates the behavior of wave packets, which are slowly modulated oscillating waves. The modulation equations describe the time evolution of the amplitude of these oscillations.
We refer to \cite{La}, Chapter 8 for an extensive description of this regime and of the many related asymptotic models it leads to.

To be more  precise, one introduces $h$ a typical depth of the fluid layer, $a$ a typical amplitude of the wave and $\lambda$ a typical wavelength in the horizontal directions (assumed to be isotropic). 
Then we set $\epsilon=a/h$ and $\mu=h^2/\lambda^2.$ We say that we are in deep water when the shallowness assumption $\mu\ll 1$ (which is made for instance in the Boussinesq-KdV regime) is not satisfied. In this situation it does not make sense to derive asymptotic models to the water waves system in terms of the parameter $\mu.$

In the {\it modulation regime}, the relevant small parameter is the {\it wave steepness}, $\varepsilon =a/\lambda$ assumed to be small. Following \cite{Tr-Dy2} it will be convenient to define finite depth, deep water and infinite depth as $\frac{\lambda}{h}$ being $O(1), O(\varepsilon)$ and $0$ respectively.
%We denote  ${\bf k}\in \R^d, d=1,2$, (respectively $k=|{\bf k}|$) the wave vector  (respectively wave number) of the wave packet. 

One first fixes a characteristic carrying wave vector ${\bf k_0}\in \R^d$, $d=1,2$ (frequently taken as the unit vector ${\bf e}_x$ in the x-direction) and considers perturbations ${\bf k}={\bf k_0}+\Delta {\bf k}.$ We will denote $k=|{\bf k}|$ and $k_0=|{\bf k_0}|$ the corresponding wave numbers. In the physical literature $\varepsilon$ is frequently expressed as $ ak.$

 In order to derive the NLS equation or higher order equations such as the Dysthe equation, one assumes a narrow bandwidth requirement, that is 
$$\delta :=\frac{|\Delta {\bf k|}}{ k_0}\ll 1,$$
 where $\Delta {\bf k}$ is the modulated wave vector. The parameter $\delta$ measures the dispersive effects. More precisely, one takes $\delta=O(\varepsilon),$ so that the dispersive effects and the wave steepness are of same order and actually one can take $\delta=\varepsilon.$

 The dispersion relation will be in particular Taylor  expanded around the wave number $k_0$.  
 At first order in $\varepsilon$ (that is neglecting the $O(\varepsilon^2)$ terms in the expansion), 
 one obtains nonlinear  Schr\"{o}dinger (NLS)  type equations or systems 
 (Benney-Roskes, Davey-Stewartson); 
 we refer to  \cite{La} Chapter 8 for a rigorous derivation.

However, as it is recalled in \cite{La}, when the wave steepness $\varepsilon$ is  not very small, the NLS approximation does not match very well with some exact computations, see for instance  Longuet-Higgins \cite{LH1, LH2}.  As recalled in \cite{CGS}, the NLS equation exhibits an unbounded region of Benjamin-Feir instability in the case of two-dimensional sideband perturbations, which extend outside the regime of a narrow-band spectrum. As a result, energy initially contained at low wave numbers can leak out to higher modes, as shown in the numerical simulations of Martin and Yuen \cite{MY}.

To overcome those shortcomings  Dysthe \cite{Dysthe}, using the method of multiscales, (see also \cite{J}) proposed in the infinite depth case ($h=+\infty$)  to include the $O(\varepsilon^2)$  terms neglected in the  derivation of the NLS equation, still  under the bandwidth assumption 
$\delta=O(\varepsilon),$
%so that the dispersive effects and the wave steepness are of same order and actually one can take $\delta=\varepsilon.$

%\textcolor{red}{est-ce qu'a la place, on ne devrait pas plutot avoir:
%$$\frac{|\Delta {\bf k}|}{k_0}=O(\varepsilon)?$$}

A consequence of the higher order expansion is the appearance of an induced  mean field $\phi$ which solves an elliptic boundary-value problem  and which can be expressed as a function of the amplitude $\psi$ via a Riesz transform.

%We consider here potential flows of an inviscid  incompressible fluid in an infinite layer, with flat bottom when the depth is finite. 

In the infinite depth case (then the dispersion relation of the water waves system is $\omega ({\bf k})=|{\bf k}|^{1/2}),$ and taking (without loss of generality) the carrier wave number ${\bf k_0}={\bf e}_x,$ one obtains the following equation in dimensionless form  for the slow envelope of the wave (see \cite{Dysthe, J} for a formal derivation and \cite{La} for  a rigorous one):\footnote{The order one expansion in $\varepsilon$ leads to the nonelliptic nonlinear Schr\"{o}dinger equation, see \cite{Z}.}

\begin{equation}\label{Dys}
\begin{split}
 &2i(\partial_t \psi+\frac{1}{2}\partial_x \psi)+\varepsilon(-\frac{1}{4}\partial_x^2+\frac{1}{2}\partial_y^2)\psi-4\varepsilon |\psi|^2\psi\\
&\qquad =\varepsilon^2\frac{i}{8}(\partial_x^3-6\partial_x\partial_y^2)\psi+2i\varepsilon^2\psi(\psi\partial_x\bar{\psi}-\bar{\psi}\partial_x \psi)
 -10\varepsilon^2i|\psi|^2\partial_x\psi-4\varepsilon^2\psi\partial_x \mathcal R_x (|\psi|^2) \, ,
\end{split}
\end{equation}
where $\mathcal R_x$ is the Riesz transform defined by $\widehat{\mathcal R _x f}(\xi,\mu)=-i\frac{\xi}{|(\xi,\mu|)|}\hat{f}(\xi,\mu).$
If $\tau$ and $(X,Y)$ are the physical variables, the derivatives $\partial_t, \partial_ x, \partial_y$ are taken with  respect to the slow variables $t=\varepsilon \tau, x=\varepsilon X, y=\varepsilon Y.$ 
Note that the relevant time scale for \eqref{Dys} is $t= O(\frac{1}{\varepsilon}),$ (and $\tau = O(\frac{1}{\varepsilon^2})$ for the physical time).

One can obtain similar equations when surface tension is taken into account (gravity-capillary waves), we refer for instance to  \cite{Ho} and to the survey article \cite {DK}. They write
\begin{equation}\label{Ho}
\begin{split}
 2i(\partial _t \psi+c_g\partial_x \psi)&+\varepsilon (p\partial_x^2+q\partial_y^2)-\varepsilon \gamma|\psi|^2\psi\\
 &=\varepsilon^2\left(-is\partial_x\partial_y^2\psi-ir\partial_x^3 \psi-iu\psi^2\partial_x\bar{\psi}+iv|\psi|^2\partial_x\psi+\psi \partial_x \mathcal R_ x (|\psi|^2)\right),
\end{split}
\end{equation}
where $c_g=\frac{\omega_0}{2k_0}\frac{1+3\kappa}{(1+\kappa)}$  is the group velocity  and  $p$, $q$, $s$, $r$, $u$ and $v$ are real parameters depending on the
surface tension parameter $\kappa\geq 0$. More precisely
$$p=\frac{3\kappa^2+6\kappa-1}{4(1+\kappa)^2},\quad q= \frac{1+3\kappa}{2(1+\kappa)},$$
$$r=-\frac{(1-\kappa)(1+6\kappa+\kappa^2)}{8(1+\kappa)^3},\quad s=\frac{3+2\kappa+3\kappa^2}{4(1+\kappa)^2},$$
$$\gamma=\frac{8+\kappa+2\kappa^2}{8(1-2\kappa)(1+\kappa)},\quad u=\frac{(1-\kappa)(8+\kappa +2\kappa^2)}{16(1-2\kappa)(1+\kappa)^2},$$
$$v=\frac{3(4\kappa^4+4\kappa^3-9\kappa^2+\kappa -8)}{8(1+\kappa)^2(1-2\kappa)^2}.$$

When $\kappa=0,$ one obtains a variant of the Dysthe equation. Note that $q$ and $s$ are strictly positive, while $p$ can achieve both signs. In particular it is negative for pure gravity waves as in the original Dysthe equation and positive
for pure capillary waves. In the latter case ($\kappa$ infinite) \eqref {Ho} becomes 
a variant of the modified Zakharov-Kuznetsov equation, i.e. 

\begin{equation} \label{Dysthe_kappa_infty}
\begin{aligned} 
2i(\partial_t\psi+\frac{3}{2}\partial_x\psi)+\frac{3}{4}\partial_x^2\psi&+\frac{3}{2}\partial_y^2\psi+\frac{1}{8}|\psi|^2\psi\\
&=-\frac{3}{4}i\partial_{xyy}^3\psi-\frac{i}{8}\partial_x^3\psi-\frac{1}{16}\psi^2\overline{\partial_x\psi}+\frac{3}{8}i|\psi|^2\partial_x\psi+\psi \partial_x \mathcal R_ x (|\psi|^2) \,.
\end{aligned}\end{equation}

\vspace{0.3cm}
The solutions of the Dysthe equation have been compared to those of the NLS equation (see \cite{LoMe1, LoMe2, HPD}) showing differences with solutions of the NLS equation, and in particular an increase of the group velocity and asymmetry of the envelope of the surface elevation with respect to the peak of the wave profile. We also refer to \cite{DT01, HCH} for comparisons with experimental data.

\begin{remark}
Stiassnie \cite{St} and Stiassnie-Shemer \cite{StSh} have derived the Dysthe equation as a particular limit of the integro-differential equation which describes  nonlinear four-wave interactions within a spectrum of surface waves introduced by Zakharov in \cite{Z}. We refer to the survey article \cite{DK} for a pedagogical description of the derivation by both methods (the multiscales and 
the Zakharov equation).
\end{remark}

In the finite depth situation, that is when the domain of the fluid is the strip $\lbrack -h<z<0 \rbrack,$ the complex envelope $\psi$ of the wave is coupled to the potential $\phi$ of the induced current 
(see \cite{LoMe1}). 
One gets (dropping the dependance in $\varepsilon)$:
\begin{equation}\label{LoMe}
\begin{cases}
\partial_t \psi+\frac{\omega}{2k_0}\partial_x \psi+i \frac{\omega}{2k_0^2}\left(\frac14\partial_x^2 \psi-\frac12\partial^2_y \psi \right)+\frac{i}{2}\omega k^2|\psi|^2\psi\\
\quad \quad \quad \quad -\frac{1}{16}\frac{\omega}{k^3}\left( \partial_x^3 \psi-6\partial_{xyy}^3 \psi \right)-\frac{\omega k_0}{4}\psi^2\partial_x \bar{\psi}+\frac{3}{2}\omega k_0 |\psi|^2\partial_x \psi+ik_0\psi\partial_x \phi_{\vert z=0}=0,\\
\partial_x^2\phi+\partial_y^2 \phi+\partial_z^2 \phi=0 \quad (-h<z<0),\\
\partial_z \phi=\frac{\omega}{2}\partial_x |\psi|^2 \quad (z=0),\\
\partial_z \phi=0 \quad (z=-h).
\end{cases}
\end{equation}

Although this paper is devoted to the two-dimensional case we comment briefly on the interesting equation arising in the one-dimensional case obtained from  \eqref{LoMe} by dropping the terms with a partial derivative in $y$. The  system for $\phi$ is easily solved by taking the Fourier transform in $x$ and we find that 
$$\partial_x \phi_{\vert z=0}=\frac{\omega}2 \partial_x \mathcal L_h (|\psi|^2)$$
where the nonlocal operator $\mathcal L_h$ is defined in Fourier variables by
$$\widehat{\mathcal L_h f}(\xi)= i\coth (h\xi)\hat f (\xi).$$
Note that in the infinite depth case $h=+\infty,$ $\mathcal L_{+\infty}$ is given by
$\widehat{\mathcal L_{+\infty} f}(\xi)= i\text{sign}(\xi)\hat f(\xi)$ that is $\mathcal L_{+\infty}=-\mathcal H$,  where $\mathcal H$ is the Hilbert transform.

We thus can write \eqref {LoMe} in the one-dimensional case as a single equation:
\begin{equation}\label{LoMebis}
\begin{split}
\partial_t \psi+&\frac{\omega}{2k_0}\partial_x \psi+i \frac{\omega}{8k_0^2}\partial_x^2 \psi+\frac{i}{2}\omega k_0^2|\psi|^2 \psi \\
&-\frac{1}{16}\frac{\omega}{k_0^3}\partial_x^3 \psi-\frac{\omega k_0}{4}\psi^2\partial_x \bar{\psi}+\frac{3}{2}\omega k_0 |\psi|^2\partial_x \psi+i\frac{\omega k_0}2\psi \partial_x \mathcal L_h  (|\psi|^2)=0.
\end{split} 
\end{equation}
This equation is reminiscent of the complex modified KdV equation 
\begin{equation}\label{cubKdV}
\partial_tu+u^2\partial_xu+\partial_x^3u=0 .
\end{equation}

In fact by eliminating the transport term in \eqref{LoMebis} by the change of variable $X=x-\frac{\omega}{2k_0}t$ (we will keep the notation $x$ for the spatial variable), then writing, with $\alpha =-\frac{1}{16}\frac{\omega}{k^3}$ and $\beta=\frac{\omega}{8k^2}$
$$\left(\xi+\frac{\beta}{3\alpha}\right )^3=\xi^3+\frac{\beta}{\alpha} \xi^2+\frac{\beta^2 \xi}{3\alpha^2}+\frac{\beta^3}{27 \alpha^3},$$
an easy computation shows that the fundamental solution of the linearization of  \eqref{LoMebis} can be expressed as 
%$$\mathfrak A(x,t)=\frac{1}{(t\alpha)^{1/3}}\exp \left(\frac{2it\beta^3}{27 \alpha^2}\right)\exp\left(\frac{-i\beta x}{2\alpha^2}\right)\int_{-\infty}^\infty e^{i\xi^3}e^{i\xi(\alpha t)^{-1/3}(x-\frac{t\beta^2}{3\alpha})}d\xi,$$
%that is 
$$\mathfrak A(x,t)=\frac{1}{(t\alpha)^{1/3}}\exp \left(\frac{2it\beta^3}{27 \alpha^2}\right)\exp\left(\frac{-i\beta x}{2\alpha^2}\right)\text{Ai}\left(\frac{1}{t^{1/3} \alpha^{1/3}}(x-\frac{\beta ^2}{3\alpha}t)\right),$$
where we have used here the Airy function  
$$\text{Ai}(z)=\int_{-\infty}^{\infty}e^{i(\xi^3+iz\xi)}d\xi.$$
It follows that the dispersive estimates for $\mathfrak A$ are essentially the same as those of the linearized KdV equation and one obtains for \eqref{LoMebis} the same results as for the complex modified KdV equation \eqref{cubKdV} (see \cite{KPV3}), that is the  local well-posedness of the initial value problem in $H^s(\R),$ $s\geq\frac{1}{4}.$

In the two-dimensional case, the mean flow potential is solution of the system (see \cite{StSh})
\begin{equation}\label{potbis}\begin{cases}
\partial_x^2 \phi+\partial_y^2 \phi+\partial_z^2 \phi=0, & -h<z<0,\\
\partial_z \phi=\frac{\omega}{2}\partial_x |\psi|^2, &z=0,\\
\partial_z \phi=0, & z=-h.
\end{cases}
\end{equation}
This system is easily solved after Fourier transform in the $(x,y)$ variables and one finds 
\begin{equation}
\partial_x \phi_{\vert z=0}=\frac{\omega}2  \partial_x \mathcal L_h (|\psi|^2), \quad\text{where}\quad
\widehat{\mathcal L_h f}(\xi,\mu)=
  i \frac{\xi}{\sqrt{\xi^2+\mu^2}}\coth\big(h\sqrt{\xi^2+\mu^2}\big) \hat f(\xi,\mu) \,,
%\frac{\omega}{2}\mathcal R_x \partial_x |\psi|^2,
\end{equation}
leading us to the following Dysthe equation in finite depth, having the same structure as the previous ones,
\begin{equation} \label{Dysthe_finite_depth} 
\begin{split}
&\partial_t \psi+\frac{\omega}{2k_0}\partial_x \psi+i \frac{\omega}{2k_0^2}\left(\frac14\partial_x^2 \psi-\frac12\partial^2_y \psi \right)+\frac{i}{2}\omega k_0^2|\psi|^2\psi\\&
\quad \quad \quad \quad -\frac{1}{16}\frac{\omega}{k_0^3}\left( \partial_x^3 \psi-6\partial_{xyy}^3 \psi \right)-\frac{\omega k_0}{4}\psi^2\partial_x \bar{\psi}+\frac{3}{2}\omega k_0 |\psi|^2\partial_x \psi+i\frac{\omega k_0}2\psi\partial_x \mathcal{L}_h (|\psi|^2)=0 \, . \\
\end{split}
\end{equation}
Observe that formally $\mathcal{L}_h \to - \mathcal{R}_x$ as $h \to +\infty$. Moreover, the symbol of the operator $\partial_x \mathcal{L}_h $ is bounded close to the origin since 
\begin{equation} \label{symbol:Dysthe_h}
i \frac{\xi^2}{\sqrt{\xi^2+\mu^2}}\coth[(\xi^2+\mu^2)^{1/2}) h] \sim  \frac{i}h \frac{\xi^2}{\xi^2+\mu^2}  
\lesssim 1 \quad \text{whenever }\   |(\xi,\mu)| \lesssim 1 \, .
\end{equation} 

\vspace{0.2cm}
%Similar equations to the Dysthe type ones occur in nonlinear optics (see for instance \cite{Z}), in particular in the
%modeling of the dynamics of femtosecond laser pulses in a nonlinear media with temporal
%dispersion. The evolution of the complex envelope $E(x, y, z, t)$ of the field is described by the
%third order NLS
%\begin{equation}\label{Zo}
%i\partial _z E+(1-i\epsilon_1\partial_t)\Delta_{\perp}E-\partial_t^2 E-i\epsilon_2 \partial_t^3 E+(1+i\epsilon_1\partial _t)g(|E|^2)E=0,
%\end{equation}
%where $\Delta_{\perp}=\partial_x^2+\partial_y^2$ is the transverse laplacian, $\epsilon_2\in \R,$ $\epsilon_1>0$ and typically  $g(z)=z.$

%As usual in nonlinear optics the evolution variable (which plays mathematically the role of time) is $z,$
%whereas $t$ will be treated as spatial variable later. The transverse Laplacian accounts for
%diffraction, while the second and third time derivatives describe group velocity and third order
%dispersion. Note that contrary to Dysthe type systems \eqref{Zo} makes sense when the spatial domain is $\R^3.$

%We plan to go back to \eqref{Zo} and to other higher order systems  in nonlinear optics (see {\it e.g.} \cite{BSNKW}) a next paper

\begin{remark}
Trulsen and Dysthe \cite{Tr-Dy2, Tr-Dy} have extended the Dysthe equation by relaxing the narrow bandwidth condition to 
$$\delta=\frac{|\Delta {\bf k}|}{k_0}=O(\varepsilon^{1/2}).$$
More precisely, in the case  $\frac{\lambda}{h}=O(\varepsilon^{1/2}),$  Trulsen and Dysthe obtained the system posed on the strip  $\R^2_{xy}\times (-h,0)_z$ (dropping the dependence on $\varepsilon)$, where the complex envelope $\psi$ of the wave is coupled to the potential $\phi$ of the induced current:
\begin{equation}\begin{cases}
\partial_t \psi+\frac{1}{2}\partial_x \psi+\frac{i}{8}\partial_x^2 \psi-\frac{i}{4}\partial_y^2 \psi-\frac{1}{16}\partial_{x}^3 \psi+\frac{3}{8}\partial_{xyy}^3 \psi \\ \quad \quad \quad \quad 
-\frac{5i}{128}\partial_x^4 \psi+\frac{15 i}{32}\partial_{xxyy}^4 \psi-\frac{3i}{32}\partial_y^4 \psi+\frac{i}{2}|\psi|^2\psi+\frac{7}{256}\partial_x^5 \psi-\frac{35}{64}\partial_{3x2y}^5 \psi\\ \quad \quad \quad \quad 
+\frac{21}{64}\partial^5_{x4y} \psi+\frac{3}{2}|\psi|^2\partial_x \psi-\frac{1}{4}\psi^2\partial_x \overline{\psi}+i\psi \partial_x \phi=0\;\;\text{at}\; z=0,\\
\partial_x^2\phi+\partial_y^2 \phi+\partial_z^2 \phi=0 \quad (-h<z<0),\\
\partial_z \phi=\frac{1}{2}\partial_x |\psi|^2 \quad (z=0),\\
\partial_z \phi=0 \quad (z=-h).
\end{cases}\end{equation}
This has the effect of adding to the Dysthe equation fourth- and fifth-order linear dispersive terms, 
while keeping the same nonlinear terms. This makes the Cauchy theory easier and this will be developed   in a subsequent paper.

As noticed in \cite{Tr-Dy2}, this equation fits better than the Dysthe type equations with the exact results in \cite{MMMSY}.
\end{remark}

\vspace{0.2cm}
In the rest of this paper we normalize out the wave steepness parameter $\varepsilon$ in \eqref{Dys} 
by the change of variables 
%\begin{equation}
%\label{eq:Tepsilon}
$(t,x,y,\psi)\mapsto \Big(\frac{t}{\varepsilon}, \frac{x}{\varepsilon}, \frac{y}{\varepsilon}, 
\frac{\psi}{\varepsilon}\Big)$
%\end{equation}
and we study the initial-value problem for  the Dysthe equation written in the form 
\begin{equation} \label{D}
\begin{split}
2i(\partial_t v+\frac12\partial_xv)&+(-\frac14\partial_x^2v+\frac12\partial_y^2v)
- \frac{i}8(\partial_x^3v-6\partial^ 3_{xyy}v) \\  &=4\ |v|^2v
+2iv(v\partial_x\bar{v}-\bar{v}\partial_xv)-10i|v|^2\partial_xv-4v\partial_x\mathcal{R}_x(|v|^2) \, ,
\end{split}
\end{equation}
where $v=v(t,x,y)$ is a complex-valued function, $t \in \mathbb R$, $(x,y)
\in \mathbb R^2$ and $\mathcal{R}_x$ denotes the Riesz transform in the $x$ variable.

Concerning the conservation laws of Dysthe type systems, one checks readily
that they conserve formally the $L^2$-norm, namely that
\begin{equation} \label{M}
M[v](t)=\int |v(t,x,y)|^2dxdy=M[v](0) \, .
\end{equation}
On the other hand and contrary to the nonlinear Schr\"odinger equation, they
do not seem to possess a conserved energy. Note however that  {\it  Hamiltonian versions} of the Dysthe equation were derived in \cite{CGS, CGNS2, CGS2} by expanding the Hamiltonian in the Zakharov-Craig-Sulem formulation of the water waves system and in \cite{GT} starting from the Zakharov integro-differential equation (\cite{Z}) enhanced with the new kernel of Krasitskii \cite{K} which yields an Hamiltonian version of the Zakharov equation.\footnote{A Hamiltonian form 
of the one-dimensional Dysthe equation is obtained in \cite{FD} via a gauge transform.}

As explained in the Introduction of \cite{GT} one reason of the non Hamiltonian form of the \lq\lq classical\rq\rq \, Dysthe equation is the choice of variable. Usually, the higher order NLS type equations are expressed in either the first-harmonic complex amplitude of the surface elevation  or in the first complex amplitude of the velocity potential evaluated at the water surface at rest.  While these variables are natural from a physical point of view, they are not optimal from a mathematical point of view since they lead to a non-Hamiltonian equation. We will study the mathematical properties of the Hamiltonian higher order NLS equations in a subsequent paper.

\subsection{Statement of main results and comments} 
We now review known results on the initial-value problem (IVP) for Dysthe type systems. 
Local well-posedness in analytic classes for a family of systems comprising the Dysthe ones has been established by A. de Bouard \cite{dB}. The local well-posedness for initial data in $H^3(\R^2)$ has been proved by Chihara \cite{Chi}, who applied techniques developed for derivative nonlinear Schr\"{o}dinger equations to these third-order equations. Koch and Saut \cite{KoSa} obtained the local well-posedness for data in $H^{\frac32}(\R^2)$ by using solely the local smoothing effects of the underlying group.

The result in \cite{KoSa} was improved very recently to $s>1$ by  Grande, Kurianski and Staffilani \cite{GrKuSta}.  The proof is based on a nontrivial adaptation of the Kenig, Ponce and Vega techniques \cite{KPV3} to this two-dimensional problem. In particular a new maximal function estimate associated to the linear part of Dysthe is proved. The authors also proved that the  IVP associated to \eqref{D} is ill-posed in $H^s(\mathbb R^2)$ for $s<0$ in the sense that the flow map data-solution cannot be 
${C}^3$ in $H^s(\mathbb R^2), s<0$.

Finally we recall  that other papers have been devoted to the initial-value problem for relevant 
third-order nonlinear Schr\"{o}dinger equations involving spatial derivatives in the nonlinearities, see for instance  \cite{La}.

The aim of this article is to improve these results by using the Fourier restriction norm method. Our first theorem proves 
global well-posedness and scattering for small initial data down to the critical space $L^2(\R^2)$.

\begin{theorem} 
\label{thm:GWPsmalldata}
%[Global well-posedness and scattering for small initial data] 
There exists $\delta>0$ such that for any $v_0\in L^2(\R^2)$ with $\|v_0\|_{L^2}<\delta$, 
there exists a unique  $v\in   \Y \subset C(\R : L^2(\mathbb R^2))$ solution to  \eqref{D} 
with  $v|_{t=0}=v_0$, 
the flow map $v_0\mapsto v$  from  $\{ v_0 \in L^2(\R^2) : \|v_0\|_{L^2}<\delta\}$ to 
  $\Y$ is smooth, and there exist $v^{\pm}_0\in L^2(\R^2)$ such that 
\begin{equation*}
v(t) -v_{\textup{lin}}^{\pm}(t) \to 0 \text{ in } L^2(\R^2) \text{, as } t\to\pm \infty\,,
\end{equation*}
where $v_{\textup{lin}}^{\pm}$ is the solution to the linear Dysthe equation starting from $v^{\pm}_0$. 
\end{theorem}

\begin{remark}
The resolution space $\Y$ is defined in Section~\ref{sect:proofmainthm} 
  -- see \eqref{defn:YYs} and \eqref{defn:Ys}. 
\end{remark}

\begin{remark}
This result is sharp due to the ill-posedness result proved in \cite{GrKuSta}.
\end{remark}

\begin{remark}
Theorem~\ref{thm:GWPsmalldata} can be seen as an analogous result to the one by 
Cazenave and Weissler in \cite{CW} for the $L^2$-critical two-dimensional cubic NLS, which is
the first-order approximation to the water wave system 
in the modulation regime and narrow bandwidth, while the Dysthe equation arises as a second-order approximation.  
\end{remark}

We also obtain a local well-posedness result in $H^s(\mathbb R^2)$, $s>0$.

\begin{theorem} \label{thm:LWP}
Let $s>0$. For any $v_0\in H^s(\mathbb R^2)$ there exist $T=T(\|v_0\|_{H^s})>0$ 
and a unique solution $v\in \X^{s,\frac12+}_T\subset C([0,T] : H^s(\mathbb R^2))$ to \eqref{D} with $v_{|_{t=0}}=v_0$. 
Moreover, for any $T' \in (0,T)$, there exists a neighborhood $\mathcal{U}$ of $v_0 \in H^s(\mathbb R^2)$ such that the flow map data-solution $w_0 \in \mathcal{U} \mapsto w \in C([0,T] : H^s(\mathbb R^2))$ is smooth. 
\end{theorem}

\begin{remark}
The well-posedness results obtained in Theorems \ref{thm:GWPsmalldata} and \ref{thm:LWP} for the Dysthe equation in dimension $d=2$ hold in lower regularity 
than the ones for the corresponding one-dimensional equation
\eqref{LoMebis}, for which well-posedness only holds in $H^s(\mathbb R)$, $s \ge \frac14$ by using the Kenig, Ponce and Vega techniques of \cite{KPV3}. This is due to the fact that the two-dimensional equation is more dispersive than its one-dimensional counterpart which allows to recover $1/4$ of a derivative instead of only $1/8$ of a  derivative in the $L^4$-Strichartz estimates. As explained below, these estimates are crucial to handling the resonant high$\times$high$\times$high frequency interactions region in the nonlinear terms. 
\end{remark}

\begin{remark} It is clear from the proofs in Section~\ref{sect:proofmainthm} 
and the bound  \eqref{symbol:Dysthe_h} that the results of Theorems \ref{thm:GWPsmalldata} and \ref{thm:LWP} also hold for the finite depth Dysthe equation \eqref{Dysthe_finite_depth}.
\end{remark}

We now discuss the main ingredients for the proofs of Theorems \ref{thm:GWPsmalldata} and \ref{thm:LWP}. First observe that by a change of variables of the form 
\begin{equation}
\label{cov}
(t,x,y,v)\mapsto (t,x+a_1t, y, e^{ia_2 x} e^{ia_3 t} v)
\end{equation}
where $a_1$, $a_2$, $a_3$ are real numbers (see Lemma~\ref{rmk:cov} below) we can remove the first- and second-order derivative linear terms on the left-hand side of \eqref{D}. 
Thus, for the sake of simplicity, we  work on the following simplified version\footnote{We prefer  not to normalize both constants in $\partial_x(\partial_x^2-3\partial_y^2)$ and thus work with a simpler Hessian determinant expression for its Fourier symbol (see \eqref{Hessian}).} of the Dysthe equation
\begin{equation} 
\label{normD}
\partial_t u+\partial_x\big(\partial_{x}^2-3\partial_{y}^2\big)u = \sum_{j=1}^ 4c_j\mathcal{N}_j(u,u,u) \, ,
\end{equation}
where $c_1, c_2,  c_3,c_4 \in\mathbb{C}$ are constants and the nonlinearities $\mathcal{N}_j(u,u,u)$ are given by
\begin{equation} \label{nonlin:D}
\begin{split}
\mathcal{N}_1(u,u,u)&=|u|^2u \, , \\
\mathcal{N}_2(u,u,u)&=|u|^2\partial_xu  \, , \\
\mathcal{N}_3(u,u,u)&=\; u^2\partial_x\bar{u} ,\\
\mathcal{N}_4(u,u,u)&=u\partial_x\mathcal{R}_x(|u|^2) \, .
\end{split}
\end{equation}
%Moreover, without loss of generality, we will fix $\alpha=1$ from now on.
We note that for $c_1=0$, the equation is invariant under the scaling transformation 
\begin{equation}
\label{scalingsym}
u(t,x,y) \mapsto u_{\lambda}(t,x,y) := \lambda u(\lambda^3 t, \lambda x, \lambda y)
\end{equation}
and since $\|u_{\lambda}(0)\|_{L^2(\R^2)} = \|u(0)\|_{L^2(\R^2)} $, we regard $L^2(\R^2)$ as the scaling critical Sobolev space of \eqref{normD} (and of \eqref{D} 
due to the invariance of the Sobolev norms under \eqref{cov}). We will prove the equivalent results of Theorems \ref{thm:GWPsmalldata} and \ref{thm:LWP} for the IVP associated to \eqref{normD} which has the advantage of having a homogeneous linear part. Theorems \ref{thm:GWPsmalldata} and \ref{thm:LWP} will then be concluded by inverting the change of variables \eqref{cov}.

The proof of the small data global well-posedness and scattering result in $L^2(\mathbb R^2)$  for \eqref{normD} relies on linear and bilinear Strichartz estimates in the context of the Fourier restriction norm method. We refer for instance to the work of Hadac, Herr and Koch \cite{HHK} and Molinet, Saut and Tzvetkov  \cite{MST} for the the KP-II equation, Gr\"unrock and  Herr \cite{GH} and Molinet, Pilod \cite{MoPi} for the Zakharov-Kuznetsov equation, Kinoshita \cite{Kin} for the modified Zakharov-Kuznetsov equation and Angelopoulos \cite{An} and Kazeykina, Mu\~noz \cite{MuKa} for the Novikov-Veselov equation. The bilinear Strichartz estimates are used to deal with the low$\times$low$\times$high and low$\times$high$\times$high frequency interactions in the nonlinearity. To handle the resonant case corresponding to the high$\times$high$\times$high frequency interactions, we use a  sharp linear Strichartz estimate proved by Kenig, Ponce and Vega in \cite{KPV}.  This estimate could also be obtained from  the dispersion estimates of Ben-Artzi, Koch and Saut in \cite{BKS} or as a corollary of a result of Carbery, Kenig, and Ziesler \cite{CKZ}  for homogeneous dispersive operators. Finally, since we are at a critical level, we need  to work in the framework of the atomic space $U^2_S $ and its dual $V^2_S $ of square bounded variation functions introduced by Koch and Tataru in \cite{KT,KT2} (see also \cite{HHK}) in order to derive the main trilinear estimate.

The proof of the arbitrarily large data local well-posedness in $H^s(\mathbb R^2)$, $s>0$, for \eqref{normD}  follows the same strategy as above. Note however that the use of $U^2-V^2$ space is not needed anymore since the problem is subcritical. It is thus enough to work with the usual $X^{s,b}$ spaces. 

It is worth noting that this method of proof only relies on the fact that  the nonlinearities $\mathcal{N}_j(u,u,u)$, $j=1,\ldots,4$ in \eqref{normD} are cubic with at most one derivative, but not on their specific structures.  Indeed, it is clear from Plancherel's identity that the linear Strichartz estimate \eqref{Strichartzlin.1} and the bilinear Strichartz estimates \eqref{BilinStrichartzI0}-\eqref{BilinStrichartzI1} also hold if one replace $\varphi$, $u_1$, $u_2$ by $\mathcal{F}^{-1}(|\widehat{\varphi}|)$, $\mathcal{F}^{-1}(|\widehat{u}_1|)$ and $\mathcal{F}^{-1}(|\widehat{u}_2|)$.

We observe that contrary to the linear Zakharov-Kuznetsov symbol, the Hessian of the dispersion relation of the Dysthe equation \eqref{normD} is sign-definite, which allows us to recover $1/4$ of a derivative in the $L^4$-Strichartz estimate 
(see \eqref{Strichartzlin.2}). A similar estimate only holds outside of cones centered at the origin for the linear Zakharov-Kuzentsov equation or with a weaker gain of $1/8$ derivatives (see the discussion in Remark~3.1 of \cite{MoPi}). This is the reason why we can reach $L^2(\mathbb R^2)$ in Theorem \ref{thm:GWPsmalldata}, while a similar result only holds for $H^s(\mathbb R^2)$, $s \ge \frac14$ for the modified Zakharov-Kuznetsov equation (see \cite{LP,RiVe,Kin}). 
In particular, 
Kinoshita constructed in \cite{Kin} a counterexample localized in the resonant  
high$\times$high$\times$high frequency interaction region  proving the sharpness of this result. 

The argument of Kinoshita \cite{Kin} could be applied to the limit version of the Dysthe equation with surface tension \eqref{Dysthe_kappa_infty} providing a local well-posedness result in $H^{\frac14}(\mathbb R^2)$ for the associated Cauchy problem. Note that a change of variable similar to the one in Lemma \ref{lem:cov} could also be applied in this case to obtain a homogeneous linear part, but this time of the Zakharov-Kuznetsov type.

\medskip
The rest of the  paper is organized as follows: in the next section we introduce the notations,  the crucial change of variable which allows to work with a homogenous linear part, define the function spaces and recall some of their important properties. 
In Section~\ref{sect:Strichartzests}, 
we recall the linear Strichartz estimates and derive the bilinear estimates. 
Those estimates are used in Section~\ref{sect:proofmainthm} 
to prove the trilinear estimates in $\R^2$. 
Finally, we conclude the paper in Section~\ref{sect:finalcomm} 
with some interesting open questions.

\section{Preliminaries and linear estimates} \label{notation}

\subsection{Notation}
For any positive numbers $a$ and $b$, the notation 
$a \lesssim b$
means that there exists a positive constant $c$ such that $a \le cb$, 
 we use $a \lesssim_{\theta} b$ when we find it necessary to 
make it explicit that the constant $c$ depends on the parameter $\theta$, 
and $a\ll b$ when $c$ is small (e.g. $c\le 10^{-2}$). 
We also write $a \sim b$ when $a \lesssim b$ and $b \lesssim a$. 
We denote by $|S|$ the Lebesgue measure of a
measurable set $S$ of $\mathbb R^d$, whereas $\# F$ denotes
the number of elements of a finite set $F$. 
The indicator function of a set $A$ is denoted by $\ind_A$.
For $u=u(t,x,y) \in \mathcal{S}(\mathbb R^{3})$, 
we use the notation $\mathcal{F}(u)$, or
$\widehat{u}$ to denote the space-time Fourier transform of $u$, whereas
$\mathcal{F}_{xy}(u)$, or $(u)^{\wedge_{xy}}$, respectively
$\mathcal{F}_t(u)=(u)^{\wedge_t}$, will denote its Fourier transform
in space, respectively in time. 

For $s \in \mathbb R$, we define $J^s$ and $D^s$, the
Bessel and Riesz potentials of order $-s$,  by
\begin{displaymath}
J^su=\mathcal{F}^{-1}_{xy}\big((1+|(\xi,\mu)|^2)^{\frac{s}{2}}
\mathcal{F}_{xy}(u)\big) \quad \text{and} \quad
D^su=\mathcal{F}^{-1}_{xy}\big(|(\xi,\mu)|^s \mathcal{F}_{xy}(u)\big).
\end{displaymath}
We also recall here that the Riesz transform $\mathcal{R}_x$ that appears in the nonlinearity of the Dysthe equation 
is given by 
\begin{equation*}
\mathcal{R}_x u =  - i  \mathcal{F}^{-1}_{xy}\left(\frac{\xi}{\sqrt{\xi^2+\mu^2}}
\mathcal{F}_{xy}(u)(\xi,\mu) \right) \,.
\end{equation*}

Let $S(t):=e^{-t \partial_x(\partial_x^2-3\partial_y^2)}$ denote the unitary group
associated with the linear part of equation \eqref{normD}, which is
to say,
\begin{equation} \label{V}
\mathcal{F}_{xy}\big(S(t)\varphi
\big)(\xi,\mu)=e^{itw(\xi,\mu)}\mathcal{F}_{xy}(\varphi)(\xi,\mu),
\end{equation}
where 
\[ w(\xi,\mu)=\xi^3-3\xi \mu^2 \, .\]
 We also define the resonance
function ${R}$ by
\begin{equation} \label{Resonance}
\begin{split}
{R}(\xi_1,\mu_1,\xi_2,\mu_2) :=&\, w(\xi_1+\xi_2,\mu_1+\mu_2)-w(\xi_1,\mu_1)-w(\xi_2,\mu_2)\\
%\mathcal{R}&(\xi_1,\mu_1,\xi_2,\mu_2) 
=&\,3\xi_1\xi_2(\xi_1+\xi_2)-3\xi_2\mu_1^2-3\xi_1\mu_2^2-6(\xi_1+\xi_2)\mu_1\mu_2\, .
\end{split}
\end{equation}

Throughout the paper, we fix a smooth cutoff function $\eta$ such that
\begin{displaymath}
\eta \in C_0^{\infty}(\mathbb R), \quad 0 \le \eta \le 1, \quad
\eta_{|_{[-5/4,5/4]}}=1 \quad \mbox{and} \quad  \mbox{supp}(\eta)
\subset [-8/5,8/5].
\end{displaymath}
For $k \in \mathbb N^{\star}=\mathbb Z \cap [1,+\infty)$, we define
\begin{displaymath}
\phi(\xi)=\eta(\xi)-\eta(2\xi), \quad
\phi_{2^k}(\xi,\mu):=\phi(2^{-k}|(\xi,\mu)|).
\end{displaymath}
and
\begin{displaymath}
\psi_{2^k}(\xi,\mu,\tau)=\phi(2^{-k}(\tau-w(\xi,\mu))).
\end{displaymath}
By convention, we also denote
\begin{displaymath}
\phi_1(\xi,\mu)=\eta(|(\xi,\mu)|), \quad \text{and} \quad
\psi_1(\xi,\mu,\tau)=\eta(\tau-w(\xi,\mu)).
\end{displaymath}
Any summations over capitalized variables such as $N, \, L$, $K$ or
$M$ are presumed to be dyadic with $N, \, L$, $K$ or $M \ge 1$,
i.e. these variables range over numbers of the form $\{2^k
: k \in \mathbb N\} $. Then, we have that
\begin{equation}
\label{defn:INintervals}
\begin{split}
 \mbox{supp} \, (\phi_N) \subset
\{\frac{5}{8}N\le |(\xi,\mu)| \le \frac{8}5N\}=:I_N& \ , \ N \ge 2\,,\\
\mbox{supp} \, (\phi_1) \subset \{|(\xi,\mu)| \le \frac85\}=:I_1&\,,
\end{split}
\end{equation}
and 
\begin{displaymath}
\sum_{N}\phi_N(\xi,\mu)=1 \,.
\end{displaymath}
Let us define the Littlewood-Paley multipliers by
\begin{equation*}
%\label{proj}
P_Nu=\mathcal{F}^{-1}_{xy}\big(\phi_N\mathcal{F}_{xy}(u)\big), \quad
Q_Lu=\mathcal{F}^{-1}\big(\psi_L\mathcal{F}(u)\big).
\end{equation*}

\subsection{Linear transformation}

\begin{lemma}
\label{lem:cov}
The function $v=v(t,x,y)$ is a solution to 
\begin{equation} \label{D:bis}
\partial_t v + \alpha_1 \partial_x (\partial_x^2- 3 \partial_y^2)v 
 +\alpha_2 i (\partial_x^2 -\partial_y^2) v +\alpha_3 \partial_x v
  =  \sum_{j=1}^ 4c_j\mathcal{N}_j(v,v,v) \,,
\end{equation}
where $\alpha_1\in \R\setminus\{0\}, \alpha_2,\alpha_3\in\R$, and $\mathcal{N}_j(v,v,v)$, 
$j=1,\ldots,4$ are the nonlinearities defined in \eqref{nonlin:D},
if and only if $u=\mathcal{T}(v)$ is a solution to 
\begin{equation}
\partial_t u + \partial_x (\partial_x^2- 3 \partial_y^2)u = \sum_{j=1}^ 4\widetilde{c}_j\mathcal{N}_j(u,u,u) \,,
\end{equation}
where
\begin{equation}
\label{cov:normD}
\mathcal{T}(v)(t,x,y):=  e^{ia_2 x} e^{-ia_3t} v(t,x+a_1t,y)
\end{equation}
and
\begin{equation*}
a_1 =  \frac{\alpha_2^2}{3\alpha_1^2} +\frac{\alpha_3}{\alpha_1}\ ,
 \ a_2=\frac{\alpha_2}{3\alpha_1} \ ,\   
a_3 = \frac{\alpha_2\alpha_3}{3\alpha_1^2} +\frac2{27}\frac{\alpha_2^3}{\alpha_1^3}\,.
\end{equation*}
%\textcolor{red}{compare with
%\begin{equation*}
%a_1 = \Big( \frac{\alpha_2}{3\alpha_1} +\alpha_3 \Big)^2\ ,
% \ a_2=\frac{\alpha_2}{3\alpha_1} \ ,\   
%a_3 = \frac{\alpha_2}{3\alpha_1} \bigg( \Big( \frac{\alpha_2}{3\alpha_1} +\alpha_3 \Big)^2  
%  + \frac{\alpha_2^2}{9\alpha_1} +1\bigg)\,.
%\end{equation*}}
\end{lemma}

\begin{remark}
\label{rmk:cov}
Before applying the transformation in Lemma \ref{rmk:cov} to equation \eqref{D} with $\epsilon=1$, we first need to rescale the $y$ variable by a factor $\sqrt{2}$ to be exactly in the configuration \eqref{D:bis}.
\end{remark}

\begin{proof}
First, we observe that the linear symbol $\omega(\xi,\mu)$ associated to 
\eqref{D:bis} can be factorized as 
\begin{align}
\label{defn:omega}
\omega(\xi,\mu) & =\alpha_1(\xi^3-3\xi\mu^2)+\alpha_2(\xi^2-\mu^2)-\alpha_3 \xi \\
\notag
& =
\alpha_1 \left( \xi+\frac{\alpha_2}{3\alpha_1}\right)^3-3\alpha_1 \left( \xi+\frac{\alpha_2}{3\alpha_1}\right)\mu^2-\left( \frac{\alpha_2^2}{3\alpha_1}+\alpha_3\right) \xi-\frac{\alpha_2^3}{27\alpha_1^2} \, .
\end{align}
Next we proceed arguing as in Section 2 in \cite{LiPiPo}. If $v$ is a solution of the linear part of \eqref{D:bis}, we first define $f(x,y,t)=e^{i\frac{\alpha_2 }{3\alpha_1}x} v(x,y,t)$ and verify that
\begin{equation*}
\partial_tf+\alpha_1 \partial_x( \partial_x^2f-3\partial^2_{y}f)+\left(\alpha_3+\frac{\alpha^2_2}{3\alpha_1}\right) \partial_xf-i\left( \frac{\alpha_2\alpha_3}{3\alpha_1} +\frac2{27}\frac{\alpha_2^3}{\alpha_1^2}  \right)f=0 \, .
\end{equation*}
Now define $g(x,y,t)=f\big(x+(\alpha_3+\frac{\alpha^2_2}{3\alpha_1})t,y,t\big)$, so that 
\begin{equation*}
\partial_tg+\alpha_1 \partial_x( \partial_x^2g-3\partial^2_{y}g)-i\left( \frac{\alpha_2\alpha_3}{3\alpha_1} +\frac2{27}\frac{\alpha_2^3}{\alpha_1^2}  \right)g=0 \, .
\end{equation*}
Finally, by defining $u(t,x,y)=e^{i\left( \frac{\alpha_2\alpha_3}{3\alpha_1} +\frac2{27}\frac{\alpha_2^3}{\alpha_1^2}  \right)\frac{t}{\alpha_1}}g(x,y,\frac{t}{\alpha_1})=\mathcal{T}(v)(t,x,y)$, we conclude that 
\begin{equation*}
\partial_tu+ \partial_x( \partial_x^2u-3\partial^2_{yy}u)=0 \, .
\end{equation*}
%\textcolor{red}{Do we get rid of the $\alpha_1$ by rescaling the time?}

Under this change of variables, one easily checks that the nonlinear terms of \eqref{D} transform as follows: 
\begin{align*}
& |v|^2 v  \mapsto   e^{-ia_2 x} e^{ia_3t} |u|^2u \,,\\
& |v|^2\partial_x v \mapsto e^{-ia_2 x} e^{ia_3t} \Big(-ia_2  |u|^2u  + |u|^2\partial_x u \Big)\,,\\
& v^2 \partial_x \overline{v} \mapsto  e^{-ia_2 x} e^{ia_3t} \Big(ia_2  |u|^2u  + u^2\partial_x \overline{u}\Big) \,,\\
& v \partial_x \mathcal{R}_x(|v|^2) \mapsto e^{-ia_2 x} e^{ia_3t} u \partial_x \mathcal{R}_x(|u|^2)\, ,
\end{align*}
which concludes the proof of the lemma.
\end{proof}

\subsection{Function spaces}

For $1 \le p \le \infty$, $L^p(\mathbb R^d)$ is the usual Lebesgue
space with the usual norm $\|\cdot\|_{L^p}$, and for $s \in \mathbb R$ ,
the Sobolev space $H^s(\mathbb R^d)$  denotes the space
of all complex-valued functions with the norm
$\|u\|_{H^s}=\|J^su\|_{L^2}.$
If $u=u(x,y,t)$ is a function defined for $(x,y) \in
\mathbb R^2$ and $t\in [0,T]$ with $T>0$, 
and if $B$
is one of the spaces defined above,  we 
define the mixed space-time spaces $L^p_TB_{xy}$,
$L^p_tB_{xy}$, $L^q_{xy}L^p_T$ by the norms
\begin{displaymath}
\|u\|_{L^p_TB_{xy}} =\Big(
\int_0^T\|u(\cdot,\cdot,t)\|_{B}^pdt\Big)^{\frac1p} \quad , \quad
\|u\|_{L^p_tB_{xy}} =\Big( \int_{\mathbb R}\|u(\cdot,\cdot,t)\|_{B}^pdt\Big)^{\frac1p},
\end{displaymath}
and
\begin{displaymath}
\|u\|_{L^q_{xy}L^p_T}= \left(\int_{\mathbb R^2}\Big( \int_0^T|u(x,y,t)|^pdt\Big)^{\frac{q}{p}}dx\right)^{\frac1q},
\end{displaymath}
for $1 \le p, \ q < \infty$ with the obvious modifications in the case $p=+\infty$ or $q=+\infty$.

%For $s$, $b \in \mathbb R$, we introduce the Bourgain spaces
%$X^{s,b}$ related to the linear part of \eqref{normD} as
%the completion of the Schwartz space $\mathcal{S}(\mathbb R^3)$
%under the norm
%\begin{equation} \label{Bourgain}
%\|u\|_{X^{s,b}} = \left(
%\int_{\mathbb{R}^3}\langle\tau-w(\xi,\mu)\rangle^{2b}\langle
%|(\xi,\mu)|\rangle^{2s}|\widehat{u}(\xi,\mu, \tau)|^2 d\xi d\mu
%d\tau \right)^{\frac12},
%\end{equation}
%where $\langle x\rangle:=1+|x|$. 
%Moreover, we define the localized-in-time version of these spaces by the norm 
%\begin{equation} \label{Bourgain}
%\|u\|_{X_T^{s,b}} = \inf \big\{ \| \widetilde{u} \|_{X^{s,b}} : \widetilde{u} \in \mathcal{S}(\mathbb R^3), 
% \widetilde{u}|_{[0,T]} = u \big\} \,.
%\end{equation}

Next, we introduce the functional framework 
introduced by Koch and Tataru in \cite{KT} %(see also \cite{HHK}) 
that allows us to reach the critical regularity for the Dysthe equation. 
We mainly work with the function spaces $U^2$ and $V^2$, 
however the machinery also requires the use of $U^p$ and $V^p$ spaces with $p\ne 2$ 
(for example in the proof of \eqref{coro:bilinStric.3} in Corollary~\ref{coro:bilinStric} below). 

\begin{definition}
Let $1\le p<\infty$ and let ${\mathcal Z} $ be the set of finite partitions $ -\infty=t_0<t_1<\cdot\cdot <t_K=+\infty $. 
For $ \{t_k\}_{k=0}^K \in {\mathcal Z}$ and $\{\phi_k\}_{k=0}^{K-1} \subset L^2(\R^2)$ with  
$\sum_{k=0}^{K-1} \| \phi_k\|_{L^2}^p =1 $ and $ \phi_0=0 $ we call the function $ a\, :\, \R \to L^2(\R^2) $ given by 
$$
a=\sum_{k=1}^K {1\! \! 1}_{[t_{k-1},t_k)} \phi_{k-1}
$$
a $ U^p$-atom and we define the atomic space 
$$
U^p:=\Bigl\{u=\sum_{j=1}^\infty \lambda_j a_j \ : \ a_j \; U^p\mbox{-atom} \mbox{ and } \lambda_j\in \R \mbox{ with } 
 \sum_{j=1}^\infty |\lambda_j | <\infty  \Bigr\} 
 $$
 with norm 
 \begin{equation}
 \label{defn:U2norm}
 \|u\|_{U^p} := \inf \Bigl\{ \sum_{j=1}^\infty |\lambda_j |\ : \  u=\sum_{j=1}^\infty \lambda_j a_j \mbox{ with }\, \lambda_j \in \R \mbox{ and } a_j \, U^p\mbox{-atom} \Bigr\} \,.
\end{equation}
 The function space $ V^p$ is defined as the normed space of all functions $ v\, :\, \R \to L^2(\R^2) $ such that 
 the limits $\lim_{t\to -\infty} v(t)$, 
 $ \lim_{t\to +\infty} v(t) $ exist, and for which the norm 
 \begin{equation}
 \label{defn:Vpnorm}
 \|v\|_{V^p}:=\sup_{\{t_k\}_{k=0}^K \in {\mathcal Z}} \Bigl( \sum_{k=1}^K \|v(t_k)-v(t_{k-1})\|_{L^2}^p \Bigr)^{1/p}
\end{equation}
 is finite, where we use the convention that 
$v(-\infty) :=  \lim_{t\to -\infty} v(t)$ and $v(+\infty) := 0$. 
\end{definition}

For basic properties of these spaces 
we refer the reader to \cite{HHK, KT}. Here we recall without proof the following results. 

\begin{lemma} 
\label{lem:UpVpembedding} 
Let $1\le p<\infty$ and 
let  $V_{-}^p$ denote the  space of all functions $v:\R\to L^2(\R^2)$ such that 
$\lim_{t\to -\infty} v(t)=0$ and $ \lim_{t\to +\infty} v(t)$ exists, endowed with the norm $\|\cdot\|_{V^p}$ given by 
\eqref{defn:Vpnorm}. 
Also, let $V_{-,\textup{rc}}^p$ denote the closed subspace of all right-continuous $V_{-}^p$-functions. 
Then, the following embeddings are continuous: 
\begin{align*}
 U^p \subset V^p_{-,\textup{rc}}\ \text{ and }\ 
 V^p_{-,\textup{rc}} \subset U^q\,,\  %\text{for any }
  p<q\,.
\end{align*}
Also, $U^p\subset U^q\subset L^{\infty}(\R; L^2(\R^2))$, 
$V^p\subset V^q$, $V_{-}^p\subset V_{-}^q$, and 
$V_{-,\textup{rc}}^p\subset V_{-,\textup{rc}}^q$, 
provided $1\le p<q<\infty$. 
\end{lemma}

\begin{lemma}[Duality lemma]
\label{lem:duality}
Let $1< p<\infty$ and $p'$ denote the H\"{o}lder conjugate, i.e.  $\frac1p+\frac{1}{p'}=1$. 
There exists a unique bilinear form $B:U^p\times V^{p'}\mapsto \C$ such that 
$V^{p'}\ni v\mapsto B(\,\cdot\,, v)\in (U^p)^*$ is an isometric isomorphism. In particular, 
\begin{align*}
\|u\|_{U^p} &= %\sup_{\substack{v\in V^2\\ \|v\|_{V^2} =1}} 
 \sup_{v\in V^{p'}\,,\, \|v\|_{V^{p'}} =1} 
 |B(u,v)| \,,\\
\|v\|_{V^p} &=  \sup_{u  \ U^{p'} \textup{\,atom}}  |B(u,v)|  \,.
\end{align*}
Moreover, if $u\in V^1_{-}$ is absolutely continuous on compact intervals, then 
$$B(u,v) = -\int_{\R} \langle u'(t), v(t) \rangle_{L^2(\R^2)} dt \,.$$ 
\end{lemma}

\medskip

Next, by recalling the notation $S(t):=e^{-t\partial_x(\partial_x^2-3\partial_y^2)}$ of the linear group associated to the linear part of the Dysthe equation \eqref{normD}, we define the spaces 
 $$ U^p_S= S(\cdot) U^p \mbox{  with norm } \|u\|_{U^p_S}=\| S(-\cdot) u\|_{U^p} $$
  $$\hspace*{-6mm} \mbox{and  }V^p_S= S(\cdot) V^p  \mbox{  with norm } 
   \|u\|_{V^p_S} =\| S(-\cdot) u\|_{V^p} \;  ,$$
and similarly  $V^p_{-,S}$ and $V^p_{-,\textup{rc},S}$. 
In particular,  by Lemma~\ref{lem:UpVpembedding}, 
any $u\in U^2_S$ also lies in $V^2_{-,\textup{rc}, S}$ and we have 
\begin{equation}
\label{eq:U2SV2Sembed}
\|u\|_{V^2_S} \lesssim \|u\|_{U^2_S} \,.
\end{equation}

The above space behaves well with respect to sharp cut-off functions since 
$\|\ind_{[a,b)} f\|_{U^2} \le \|f\|_{U^2}$ and 
$\|f\|_{V^2} \le \|\ind_{[a,b)} f\|_{V^2} \le 2\|f\|_{V^2}$, 
for any $-\infty<a<b\le +\infty$. 

A simple application of the duality Lemma~\ref{lem:duality} 
provides the following linear non-homogeneous estimate. 

\begin{lemma}
\label{lem:dualityS}
Let $1< p,p'<\infty$ with  $\frac1p+\frac{1}{p'}=1$, and
 $-\infty<a<b\le +\infty$.
  If $f\in L^1([a,b); L^2(\R^2))$, then 
\begin{equation*} 
\label{est:nhlinU}
\Bigl\|  \ind_{[a,b)}(t)  \int_0^t S(t-t') f(t') \, dt' \Bigr\|_{U^p_S} \lesssim 
  \sup_{\|g\|_{V^{p'}_S}=1} \bigg| \int_a^b \int_{\R^2} f g \bigg| \,.
\end{equation*}
\end{lemma}
%\begin{proof}
%By the duality lemma \ref{lem:duality}, we have 
% \begin{align*}
%\Big\| \int_0^t S(t-t') f(t') \Big\|_{U^2_S} 
%&= \Big\| \int_0^t S(-t')f(t') \Big\|_{U^2} 
%= \sup_{\|v\|_{V^2}=1} \Big| \int_{\R} \langle S(-t)f(t) \,,\,v\rangle_{L^2(\R^2)}  \Big| \\
%&= \sup_{\|g\|_{V_S^2}=1} \Big| \int_{\R^3}fg\, dxdy\,dt \Big| \,.
% \end{align*}
%\end{proof}

\medskip
 
 We are now ready  to define our resolution space for constructing solutions to \eqref{normD}. 
 We denote by $Y^{s} $ the closure of all functions 
 $ u\in C(\R : H^s(\R^2))$ with respect to the norm
 \begin{equation}
 \label{defn:Ys}
 \|u\|_{Y^{s}} := \Bigl( \sum_{N} N^{2s} \| P_{N}u\|_{U^2_S}^2 \Bigr)^{1/2} <\infty \,.
\end{equation}
 
% \noindent
%Let $I$ be a  time interval and 
%%$B=Y^s$ or $B=X^{s,b}$. Then, if $u: \mathbb R^2 \times
%%[0,T] \rightarrow \mathbb C$, 
%we define the localized-in-time space $Y^s(I)$ through the norm
%\begin{displaymath}
%\|u\|_{Y^s(I)}=\inf \left\{\|\tilde{u}\|_{Y^s} \ : \ \tilde{u}\in \mathcal{S}(\R^3), \
%\tilde{u}_{|_{\mathbb R^2 \times I}} = u\right\}.
%\end{displaymath}

%\subsection{Linear and multilinear estimates in the $U^2_S-V^2_S$ setting}
%We summarize here                                                                                                                               
%the main properties of these spaces (see \cite{HHK, KT} for detailed proofs). 

%The first one states the linear homogenous and non-homogeneous estimates in the $U^2_S$ spaces. 
% \begin{proposition} \label{prop:linUV}
%For any  $ \psi \in C^\infty_c(\R)$, we have 
%\begin{equation} \label{est:nhlinU}
%\Bigl\|\psi(t)\int_0^t S(t-t') f(t',\cdot) \, dt' \Bigr\|_{U^2_S} \lesssim \sup_{\|v\|_{V^2_S}=1} \Bigl| \int_{\mathbb R^3} f v \Bigr| , \quad  \forall f \in C^\infty_c(\mathbb R^3) \; .
%\end{equation}
% \end{proposition}
 
 The following proposition is a transference principle in the $U^p_S$ setting. 
 \begin{proposition}
[{\cite[Proposition~2.19.(i)]{HHK}}]
 \label{prop:transf}
 Let $ \mathcal{L}_0 $ : $ L^2(\mathbb R^3)\times \cdot\cdot \cdot \times L^2(\mathbb R^3) \to L^1_{\textup{loc}}(\mathbb R^3) $ be a $n$-linear operator. Assume that for some $ 1\le p,q\le \infty $, we have
 $$
 \| \mathcal{L}_0(S(\cdot) \phi_1,\cdot\cdot\cdot, S(\cdot) \phi_n) \|_{L^p_t(\mathbb R: L_{x,y}^q(\mathbb R^2))} \lesssim 
 \prod_{i=1}^n \| \phi_i\|_{L^2} \; .
 $$
 Then, 
 there exists $ \mathcal{L}\, :\, U^p_S \times \cdot\cdot \cdot \times U^p_S \to L^p_t(\R: L_{x,y}^q(\R^2)) $ 
 satisfying 
  $$
 \| \mathcal{L}(u_1,\cdot\cdot\cdot,  u_n) \|_{L^p_t(\mathbb R: L_{x,y}^q (\mathbb R^2))} \lesssim 
 \prod_{i=1}^n \| u_i\|_{U^p_S} 
 $$
 such that $\mathcal{L}(u_1,\cdot\cdot\cdot, u_n)(t)(x,y)=\mathcal{L}_0(u_1(t),\cdot\cdot\cdot, u_n(t))(x,y) $ almost everywhere. 
 \end{proposition}
 
% \textcolor{red}{We will need more estimates with the space $V_S^2$ here...}

\medskip 

For the proof of Theorem~\ref{thm:LWP}, we use Bourgain spaces related to the linear part of \eqref{normD}. 
Thus, for any  $s, b \in \R$, let 
$X^{s,b}$ denote the completion of the Schwartz space $\mathcal{S}(\mathbb R^3)$
under the norm
\begin{equation} \label{Bourgain}
\|u\|_{X^{s,b}} = \bigg( \sum_{N,L} N^{2s} L^{2b} \|P_NQ_L u\|_{L^2(\R^3)}^2 \bigg)^{\frac12},
\end{equation}
and we define the localized-in-time version of these spaces by the norm 
\begin{equation} \label{Bourgain}
\|u\|_{X_T^{s,b}} = \inf 
  \big\{ \| \widetilde{u} \|_{X^{s,b}} : \widetilde{u} \in \mathcal{S}(\mathbb R^3), 
   \widetilde{u}|_{[0,T]} = u \big\}
  \,.
\end{equation}

The basic properties of these spaces that we need here are stated in the following proposition 
(see  e.g. \cite{TaoCBMS07}). 

\begin{proposition}
\label{prop:basicXsb}
Let $s\in\R$ and $\chi\in C_0^{\infty}(\R)$ be a time cutoff function. 

\noindent
\textup{(i)} For any $b>\frac12$, 
\begin{equation}
\|\chi(t) S(t) \varphi \|_{X^{s,b}} \lesssim \|\varphi\|_{H^s(\R^2)}\,.
\end{equation}

\noindent
\textup{(ii)}
For any $0<\delta<\frac12$, 
\begin{equation}
\bigg\|\chi(t) \int_0^t S(t-t')f(t') \bigg\|_{X^{s,-\frac12+\delta}} \lesssim \big\|f \big\|_{X^{s,\frac12+\delta}} \,.
\end{equation}

\noindent
\textup{(iii)} For any $T>0$ and $-\frac12< b'\le b < \frac12$, 
\begin{equation}
\big\|\chi(t/T) u \big\|_{X^{s,b'}} \lesssim T^{b-b'} \big\|f \big\|_{X^{s,b}} \,.
\end{equation}

\noindent
\textup{(iv)} 
For any $T>0$ and $ b > \frac12$, we have $X_T^{s,b}\subset C([0,T]: H^s(\R^2))$. 

\end{proposition}

\medskip

\section{Linear and bilinear Strichartz estimates}
\label{sect:Strichartzests}

\subsection{Linear Strichartz estimate}

Let $\Omega$ be a homogeneous polynomial in $\mathbb R^2$. 
Let $H_{\Omega}$ be the Hessian determinant of $\Omega$, \textit{i.e.} 
\begin{equation} \label{CKZ2}
H_{\Omega}(\xi,\mu) = \det \mathrm{Hess} \, \Omega(\xi,\mu) \, .
\end{equation}
We denote by  $|H_{\Omega}(D_x,D_y)|^{\frac18}$ and $e^{it\Omega(D_x,D_y)}$ the Fourier
multipliers associated to $|H_{\Omega}(\xi,\mu)|^{\frac18}$ and
$e^{it\Omega(\xi,\mu)}$, \textit{i.e.}
\begin{equation} \label{CKZcoro1}
\mathcal{F}_{xy}\Big(|H_{\Omega}(D_x,D_y)|^{\frac18}\varphi
\Big)(\xi,\mu)= |H_{\Omega}(\xi,\mu)
|^{\frac18}\mathcal{F}_{xy}(\varphi)(\xi,\mu)
\end{equation}
and
\begin{equation} \label{CKZcoro2}
\mathcal{F}_{xy}\big(e^{it\Omega(D_x,D_y)}\varphi
\big)(\xi,\mu)=e^{it\Omega(\xi,\mu)}\mathcal{F}_{xy}(\varphi)(\xi,\mu).
\end{equation}

As a consequence of Carbery, Kenig and Ziesler in \cite{CKZ}, we have the following sharp $L^4$-Strichartz estimate for the unitary group $e^{it\Omega(D_x,D_y)}$ (see for example the proof of Corollary 3.4 in \cite{MoPi}).
\begin{theorem} \label{CKZcoro}
For all $\varphi \in L^2(\mathbb R^2)$, we have 
\begin{equation} \label{CKZcoro3}
\big\| |H_{\Omega}(D_x,D_y)|^{\frac18}e^{it\Omega(D)}\varphi\big\|_{L^4_{xyt}} \,.
\lesssim \|\varphi\|_{L^2},
\end{equation}
\end{theorem}

Now, we apply Theorem \ref{CKZcoro} in the case of the unitary group $e^{-t\partial_x(\partial_x^2-3\partial_y^2)}$.
\begin{proposition} \label{Strichartzlin}
\textup{(i)} 
For all $\varphi \in L^2(\mathbb R^2)$, 
we have
\begin{equation} 
\label{Strichartzlin.1}
\big\|
|D^{\frac14}e^{-t\partial_x(\partial_x^2-3\partial_y^2)}\varphi\big\|_{L^4_{x,y,t}(\R^3)}
\lesssim \|\varphi\|_{L^2(\R^2)} \,.
\end{equation}

\noindent
\textup{(ii)} 
Let $N$ be a dyadic number in $\{ 2^k  :  k \in \mathbb N^{\star} \} $. 
For all  $u \in V^2_{-,S}$ we have 
\begin{equation} 
\label{Strichartzlin.2}
\big\| P_Nu\big\|_{L^4(\R^3)} \lesssim N^{-\frac14}
\|P_Nu\|_{V^2_S} \,.
\end{equation}
\end{proposition}

\begin{proof} The symbol associated to $e^{-t\partial_x(\partial_x^2-3\partial_y^2)}$ is
given by $w(\xi,\mu)=\xi^3-3\xi\mu^2$. An easy computation shows that
\begin{equation}
\label{Hessian} 
H_{w}(\xi,\mu):=\det \mathrm{Hess} \, w(\xi,\mu)=-36(\xi^2+\mu^2).
\end{equation}
Estimate \eqref{Strichartzlin.1} follows then as a direct application
of Corollary \ref{CKZcoro} and estimate \eqref{Strichartzlin.2} follows from Proposition \ref{prop:transf}.
\end{proof}

\begin{remark} In the case of the linear Dysthe group $e^{-t\partial_x(\partial_x^2-3\partial_y^2)}$, Theorem \ref{CKZcoro} follows actually from the dispersive estimates proved by Kenig, Ponce and Vega in \cite{KPV}. Indeed the symbol  $w(\xi,\mu)=\xi^3-3\xi\mu^2$ clearly satisfies the conditions of Lemma 3.4 in \cite{KPV} on $\mathbb R^2$, which provides the sharp decay estimate (see also Ben-Artzi, Koch and Saut in \cite{BKS}) 
\begin{equation*}
\left| \int_{\mathbb R^2}e^{i(tw(\xi,\mu)+x\xi+y\mu)}H_w(\xi,\mu) d\xi d\mu\right| \lesssim |t|^{-1} \, , \quad \forall \, (x,y) \in \mathbb R^2, \ \forall \, t \in \mathbb R \, .
\end{equation*} 
Moreover, since $w(\xi,\mu)$ is homogeneous, estimate \eqref{Strichartzlin.1} is then deduced from \eqref{Hessian} arguing as in Theorem 3.1 and Remark (a) in \cite{KPV}.
\end{remark}

\subsection{Bilinear Strichartz estimates}

In this subsection, we state crucial bilinear
estimates related to the Dysthe dispersion relation for functions
defined on $\mathbb R^3$.

\begin{proposition} \label{BilinStrichartzI}
Let $N_1, \ N_2, \ L_1, \ L_2$ be dyadic numbers in $\{ 2^k  :  k \in \mathbb N \}$ and 
assume that $u_1, u_2$ are two functions
in $L^2(\mathbb R^3)$. 
%Then,
%\begin{equation} 
%\label{BilinStrichartzI0}
%\begin{split}
%&\|(P_{N_1}Q_{L_1}u_1)(P_{N_2}Q_{L_2}u_2)\|_{L^2(\R^3)} \\
%&\quad \qquad \lesssim\,
%\min\{L_1 ,L_2\}^{\frac12} \min\{N_1 , N_2\}\|P_{N_1}Q_{L_1}u_1\|_{L^2(\R^3)}
%\|P_{N_2}Q_{L_2}u_2\|_{L^2(\R^3)}
%\end{split}
%\end{equation}
%Assume moreover that 

\noindent
\textup{(i)}
We have the basic estimate
\begin{equation} 
\label{BilinStrichartzI0}
\begin{split}
&\|(P_{N_1}Q_{L_1}u_1)(P_{N_2}Q_{L_2}u_2)\|_{L^2(\R^3)}\\
&\qquad\quad  \lesssim \,
   \min\{N_1, N_2\} \min\{L_1 ,L_2\}^{\frac12} 
    \|P_{N_1}Q_{L_1}u_1\|_{L^2(\R^3)} \|P_{N_2}Q_{L_2}u_2\|_{L^2(\R^3)} \,.
\end{split}
\end{equation}

\noindent
\textup{(ii)} If $N_1 \ge 4N_2$, then 
\begin{equation} \label{BilinStrichartzI1}
\|(P_{N_1}Q_{L_1}u_1) (P_{N_2}Q_{L_2}u_2)\|_{L^2(\R^3)}  \lesssim
N_2^{\frac12}N_1^{-1} L_1^{\frac12}L_2^{\frac12}\|P_{N_1}Q_{L_1}u_1\|_{L^2(\R^3)}
\|P_{N_2}Q_{L_2}u_2\|_{L^2(\R^3)}.
\end{equation}

\end{proposition}

The proof of Proposition \ref{BilinStrichartzI} is given in Lemma 9 in \cite{An} (see also Proposition 3.6 in \cite{MoPi}). For the sake  of completeness we give the argument here. The proof relies on some basic Lemmas stated for example in
\cite{MST}.

\begin{lemma} \label{basicI}
Consider a set $\Lambda \subset \mathbb R \times R$. Let the projection on the $\mu$ axis
be contained in a set $I \subset \mathbb R$. Assume in addition that
there exists $C>0$ such that for any fixed $\mu_0 \in I$, $
\left| \{\xi \in \mathbb R  :  (\xi,\mu_0) \in \Lambda\}\right| \le C$. Then, we get that 
$|\Lambda| \le C |I|$.
\end{lemma}

The second one is a direct consequence of the mean value theorem.
\begin{lemma} \label{basicII}
Let $I$ and $J$ be two intervals on the real line and $f:J
\rightarrow \mathbb R$ be a smooth function. Then,
\begin{equation} \label{basicII1}
 \left| \{x \in J \ : \ f(x) \in I\} \right| \le \frac{|I|}{\inf_{\xi
\in J}|f'(\xi)|}.
\end{equation}
%\begin{equation} \label{basicII2}
%\# \{q \in J \cap \mathbb Z  \ : \ f(q) \in I\} \le
%\frac{|I|}{\inf_{\xi \in J}|f'(\xi)|}+1.
%\end{equation}
\end{lemma}

%In the case where $f$ is a polynomial of degree $2$, we also have the following result.
%\begin{lemma} \label{basicIII}
%Let $a \neq 0, \ b, \ c$ be real numbers and $I$ be an interval on
%the real line. Then,
%\begin{equation} \label{basicIII1}
%\left| \{x \in J \ : \ ax^2+bx+c \in I\} \right| \lesssim
%\frac{|I|^{\frac12}}{|a|^{\frac12}}.
%\end{equation}
%\end{lemma}

\medskip

\begin{proof}[Proof of Proposition \ref{BilinStrichartzI}]
The Cauchy-Schwarz inequality and Plancherel's identity yield
\begin{equation} \label{BilinStrichartzI3}
\begin{split}
\|(P_{N_1}Q_{L_1}u_1)(P_{N_2}&Q_{L_2}u_2)\|_{L^2(\R^3)} \\ &=
\|(P_{N_1}Q_{L_1}u_1)^{\wedge} \star (P_{N_2}Q_{L_2}u_2)^{\wedge}\|_{L^2(\R^3)} \\ & \lesssim
\sup_{(\xi,\mu,\tau) \in \mathbb R^3}|A_{\xi,\mu,\tau} |^{\frac12}
\|P_{N_1}Q_{L_1}u_1\|_{L^2(\R^3)}\|P_{N_2}Q_{L_2}u_2\|_{L^2(\R^3)} \,,
\end{split}
\end{equation}
where
\begin{displaymath}
\begin{split}
A_{\xi,\mu,\tau} :=&\Big\{  (\xi_1,\mu_1,\tau_1) \in \mathbb R^3 \ : \
 |(\xi_1,\mu_1)| \in I_{N_1}, \
|(\xi-\xi_1,\mu-\mu_1)| \in I_{N_2} \\ & \quad \quad
|\tau_1-w(\xi_1,\mu_1)| \in I_{L_1}, \
|\tau-\tau_1-w(\xi-\xi_1,\mu-\mu_1)| \in I_{L_2}\Big\} \ .
\end{split}
\end{displaymath}
It remains then to estimate the measure of the set
$A_{\xi,\mu,\tau}$ uniformly in $(\xi,\mu,\tau) \in \mathbb R^3$. 

To obtain \eqref{BilinStrichartzI0}, we use the trivial estimate
\begin{displaymath} 
|A_{\xi,\mu,\tau} | \lesssim \min\{L_1 ,L_2\} \min\{N_1, N_2\}^2, 
\end{displaymath} 
for all $(\xi,\mu,\tau) \in \R^3$.

Now we turn to the proof of estimate \eqref{BilinStrichartzI1}. First, we get easily 
from the triangle inequality that 
\begin{equation} \label{BilinStrichartzI4}
|A_{\xi,\mu,\tau} | \lesssim \min\{ L_1 , L_2\} |B_{\xi\,\mu,\tau} | \,,
\end{equation}
where 
\begin{equation}  \label{BilinStrichartzI40}
\begin{split}
B_{\xi,\mu,\tau} :=&\, \Big\{  (\xi_1,\mu_1) \in \mathbb R^2 \ : \
 |(\xi_1,\mu_1)| \in I_{N_1}, \
|(\xi-\xi_1,\mu-\mu_1)| \in I_{N_2} \\ & \quad \quad
|\tau-w(\xi,\mu) - {R}(\xi_1,\mu_1,\xi-\xi_1,\mu-\mu_1)| \lesssim \max\{L_1 ,L_2\} \Big\} 
\end{split}
\end{equation}
and where ${R}(\xi_1,\mu_1,\xi_2,\mu_2)$ is the resonance function defined in \eqref{Resonance}. 

In the case where $|\xi_1| \gg |\mu_1|$ or $|\mu_1| \gg |\xi_1|$, we observe from the hypothesis  $N_1 \ge 4N_2$  that 
\begin{displaymath} 
\Big| \frac{\partial {R}}{\partial \xi_1}(\xi_1,\xi-\xi_1,\mu_1,\mu-\mu_1)\Big|=
\big|3(\xi_1^2-\mu_1^2)-3((\xi-\xi_1)^2-(\mu-\mu_1)^2) \big| \gtrsim N_1^2 \  .
\end{displaymath}
Then, if we define
$B_{\xi,\mu,\tau}(\mu_1) =\{ \xi_1 \in \mathbb R \ : \ (\xi_1,\mu_1) \in B_{\xi,\mu,\tau}\}$, we deduce applying estimate \eqref{basicII1} that 
\begin{displaymath}  
|B_{\xi,\mu,\tau}(\mu_1) | \lesssim \frac{\max\{L_1 , L_2\} }{N_1^2} \ ,
\end{displaymath}
for all $\mu_1 \in \mathbb R$. Thus, it follows from Lemma \ref{basicI} that 
\begin{equation} \label{BilinStrichartzI5} 
|B_{\xi,\mu,\tau} | \lesssim \frac{N_2}{N_1^2} \max\{L_1 , L_2\} \,.
\end{equation}

In the case where $|\xi_1| \sim |\mu_1|$, then we use that 
\begin{displaymath} 
\Big| \frac{\partial {R}}{\partial \mu_1}(\xi_1,\xi-\xi_1,\mu_1,\mu-\mu_1)\Big|=
\big|6\xi_1\mu_1-6(\xi-\xi_1)(\mu-\mu_1) \big| \gtrsim N_1^2 \  .
\end{displaymath}
Then, if we define
$B_{\xi,\mu,\tau}(\xi_1) =\{ \mu_1 \in \mathbb R \ : \ (\xi_1,\mu_1) \in B_{\xi,\mu,\tau}\}$, we deduce applying estimate \eqref{basicII1} that 
\begin{displaymath}  
|B_{\xi,\mu,\tau}(\xi_1) | \lesssim \frac{\max\{L_1 , L_2\} }{N_1^2} \,,
\end{displaymath}
for all $\xi_1 \in \mathbb R$, so that estimate \eqref{BilinStrichartzI5} also follows in this case. 

Finally, we conclude the proof of the estimate \eqref{BilinStrichartzI1} 
by gathering estimates \eqref{BilinStrichartzI3}--\eqref{BilinStrichartzI5}. 

\end{proof}

%\begin{remark} \label{rem:bilinStric} 
%Note that 
%Proposition~\ref{BilinStrichartzI} also holds true if we replace 
%$\|(P_{N_1}Q_{L_1}u_1)(P_{N_2}Q_{L_2}u_2)\|_{L^2}$ by 
%$\|(\overline{P_{N_1}Q_{L_1}u_1})(P_{N_2}Q_{L_2}u_2)\|_{L^2}$ or $\|(P_{N_1}Q_{L_1}u_1)(\overline{P_{N_2}Q_{L_2}u_2})\|_{L^2}$ on the left-hand side of \eqref{BilinStrichartzI1} 
%(and similarly for \eqref{BilinStrichartzI0}). 
%% Indeed, it is clear from the proof since 
%%\[ \left(\, \overline{P_{N}Q_{L}u}\, \right)^{\wedge}(\xi,\mu,\tau)= \overline{(P_{N}Q_{L}u)^{\wedge}(-\xi,-\mu,-\tau)} \]
%%and 
%%\[ |-\tau-w(-\xi,-\mu)| \in I_N \quad \Leftrightarrow \quad |\tau-(\xi^3-3\xi\mu^2)+(\xi^2+\mu^2)| \in I_N \]
%\end{remark}

\begin{corollary} 
\label{coro:bilinStric}
Let $N_1, \ N_2$ be dyadic numbers in $\{ 2^k  :  k \in \mathbb N \}$. 

\noindent
\textup{(i)}
If $N_1 \ge 4N_2$, then 
\begin{align} 
\label{coro:bilinStric.1}
\| P_{N_1}u_1P_{N_2}u_2\|_{L^2(\R^3)}  &\lesssim
N_2^{\frac12}N_1^{-1} \|P_{N_1}u_1\|_{U^2_S}
\|P_{N_2}u_2\|_{U^2_S} \, ,\  
\text{for any } u_1, \, u_2 \in U^2_S\,,\ \text{and}\\
\label{coro:bilinStric.3}
\| P_{N_1}u_1P_{N_2}u_2\|_{L^2(\R^3)}  &\lesssim
N_2^{\frac12}N_1^{-1} \ln^2\Big(\frac{N_1}{N_2}\Big)\|P_{N_1}u_1\|_{V^2_S}
\|P_{N_2}u_2\|_{V^2_S} \, ,\ 
\text{for any } u_1, \, u_2 \in V^2_{-,S}\,.
\end{align}

\noindent
\textup{(ii)} If $N_1 , N_2\lesssim 1$, we have 
\begin{align} 
\label{coro:bilinStric.5}
\| P_{N_1}u_1P_{N_2}u_2\|_{L^2(\R^3)}  &\lesssim
 \|P_{N_1}u_1\|_{V^2_S} \|P_{N_2}u_2\|_{V^2_S} \, ,\ 
\text{for any } u_1, \, u_2 \in V^2_{-,S}\,.
\end{align}
\end{corollary}

 \begin{proof}
 Let $ \chi \in C^\infty_c(\R) $ be a cut-off function satisfying $0 \le \chi \le 1$, $\chi_{|_{[-1,1]}} =1$ and $\text{supp} \, \chi \in [-2,2]$. We deduce by using Plancherel's identity that for $ \varphi \in L^2(\mathbb R^2) $ and any dyadic number $ L \ge 1 $ 
\begin{equation*}
\mathcal{F}\big( Q_L \chi S(\cdot)\varphi \big)(\xi,\mu,\tau) 
 = \phi_L(\tau-w(\xi,\mu)) \widehat{\chi}(\tau - w(\xi,\mu)) \mathcal{F}_{xy}(\varphi)(\xi,\mu) \, .
\end{equation*}
Then by using again Plancherel's identity and taking into account that $\widehat{\chi}$ is rapidly decreasing, it holds that 
\begin{equation*}
 \|Q_L\chi S(\cdot) \varphi \|_{L^2_{x,y,t}}  \le \| \phi_L \widehat{\chi}\|_{L^2(\R)}\| \varphi \|_{L^2(\R^2)} 
  \lesssim L^{-4} \|\varphi \|_{L^2(\R^2)} \, .
\end{equation*}
This ensures by using \eqref{BilinStrichartzI1} that 
\begin{equation} \label{coro:bilinStric.2}
\begin{split}
\|(P_{N_1} &S(\cdot) \varphi_1)( P_{N_2} S(\cdot) \varphi_2)\|_{L^2([-1,1]\times \mathbb R^2)} \\ & \le \|(\chi^{\frac12} P_{N_1} S(\cdot) \varphi_1)( \chi^{\frac12} P_{N_2} S(\cdot) \varphi_2) \|_{L^2(\R^3)} \lesssim N_2^{\frac12}N_1^{-1}\| P_{N_1} \varphi_1\|_{L^2(\R^2)} \| P_{N_2}\varphi_2\|_{L^2(\R^2)} \; .
\end{split}
\end{equation}
We can remove the time cut off $\chi$ by using the scaling invariance \eqref{scalingsym}. Indeed let $u_1=S(t)\varphi_1$ and $u_2=S(t)\varphi_2$. Then it follows by using the notation in \eqref{scalingsym} that for any $\lambda>0$
\begin{displaymath}
\|P_{N_1}u_1 P_{N_2}u_2 \|_{L^2([-\lambda^3,\lambda^3]\times \mathbb R^2)} =\lambda^{\frac12}\|(P_{N_1}u_1)_{\lambda}( P_{N_2} u_2)_{\lambda}\|_{L^2([-1,1]\times \mathbb R^2)} \, .
\end{displaymath}
Since $(P_{N}u)_{\lambda}=P_{\lambda N}(u)_{\lambda}$, we conclude from \eqref{coro:bilinStric.2} that 
\begin{displaymath} 
\|P_{N_1} S(\cdot) \varphi_1)( P_{N_2} S(\cdot) \varphi_2)\|_{L^2([-\lambda^3,\lambda^3]\times \mathbb R^2)}  \lesssim N_2^{\frac12}N_1^{-1}\| P_{N_1} \varphi_1\|_{L^2(\R^2)} \| P_{N_2}\varphi_2\|_{L^2(\R^2)} \; ,
\end{displaymath}
which implies by letting $\lambda$ to $+\infty$, 
\begin{equation} 
\|P_{N_1} S(\cdot) \varphi_1)( P_{N_2} S(\cdot) \varphi_2)\|_{L^2( \mathbb R^3)}  \lesssim N_2^{\frac12}N_1^{-1}\| P_{N_1} \varphi_1\|_{L^2(\R^2)} \| P_{N_2}\varphi_2\|_{L^2(\R^2)} \; .
\end{equation}

Estimate \eqref{coro:bilinStric.1} follows then by applying Proposition \ref{prop:transf}, while 
\eqref{coro:bilinStric.3} follows as in \cite[Corollary~2.21]{HHK}. We obtain \eqref{coro:bilinStric.5} analogously.  

\end{proof}
 
%\noindent  \textcolor{blue}{\underline{Question}: I am not sure that we need a function $\chi$ in estimate \eqref{coro:bilinStric.2} and thus in \eqref{coro:bilinStric.1}? }

%\begin{remark} 
%Similarly to Remark \ref{rem:bilinStric}, Corollary \eqref{coro:bilinStric} also holds true if we replace 
%$\|\chi^2P_{N_1}u_1P_{N_2}u_2\|_{L^2_{x,y,t}}$ by $\|\chi^2\overline{P_{N_1}u_1}P_{N_2}u_2\|_{L^2_{x,y,t}}$ or $\|\chi^2P_{N_1}u_1\overline{P_{N_2}u_2}\|_{L^2_{x,y,t}}$ on the right-hand side of \eqref{coro:bilinStric.1}. 
%\end{remark}

%By using the $V^2_S$-norm, the $L^4$-Strichartz and the bilinear estimates read as follows 
%and the proof  follows that of \cite[Corollary~2.21]{HHK}.

\begin{corollary}
For $0< \theta\le \frac14$, we have 
\begin{equation} 
\label{Strichartzlin.4interp}
\big\| P_Nu\big\|_{L^4(\R^3)} \lesssim_{\theta} N^{-\frac14 +\frac{3\theta}{1+2\theta}} \|P_Nu\|_{X^{0,\frac12-\theta}} \,.
\end{equation}
\end{corollary}
\begin{proof}
By using the transference principle for $X^{s,b}$-spaces (see e.g. \cite[Lemma~2.9]{TaoCBMS07}), from \eqref{Strichartzlin.1} we have 
\begin{equation}
\label{Strichartzlin.4trpr}
\|P_N u\|_{L^4(\R^3)} \lesssim_{\theta} N^{-\frac14} \|P_Nu\|_{X^{0,\frac12+\theta}} \,.
\end{equation}
The estimate \eqref{Strichartzlin.4interp} follows 
by interpolating \eqref{Strichartzlin.4trpr} 
with the trivial $L^4$-estimate that follows from \eqref{BilinStrichartzI0}, namely with 
\begin{equation*}
\|P_N u\|_{L^4(\R^3)} \lesssim N^{\frac12}\|P_Nu\|_{X^{0,\frac14}} \,. 
\end{equation*}
\end{proof}

%\smallskip

\section{Proof of the main results}
\label{sect:proofmainthm}

For $T\in (0,+\infty]$, we consider 
\begin{equation}
I_{T,j}(u_1,u_2,u_3)(t) :=  \ind_{[0,T)}(t)  \int_{0}^t   S(t-t') \mathcal{N}_j(u_1, u_2, u_3)(t') \,dt'   \,.
\end{equation}

 \begin{proposition}
 \label{prop:trilinests}
 Let  $s\ge 0$ 
 and $u_1,u_2,u_3\in Y^s$.  
 %$\chi\in C_0^{\infty}(\R)$ be a non-negative time cut-off function. 
 
 \noindent
 \textup{(i)} 
 There exists $C>0$ such that 
 for all $0<T< +\infty $  and 
 $1\le j\le 4$, we have
 \begin{equation}
 \label{keytrilinests}
% \bigg\|  \ind_{[0,T)}(t) \int_{0}^t S(t-t') \mathcal{N}_j(u_1, u_2, u_3)(t') \,dt'   \bigg\|_{Y^s} 
% \le C \prod_{i=1}^3 \|u_i\|_{Y^s} \,.
 % \|u_1\|_{U^2_S}  \|u_2\|_{U^2_S}  \|u_3\|_{U^2_S}  \|u_4\|_{V^2_S} \,,
 \big\| I_{T,j}(u_1,u_2,u_3)\big\|_{Y^s} \le C   \|u_1\|_{Y^s}  \|u_2\|_{Y^s}  \|u_3\|_{Y^s} \,.
 \end{equation}
 
  \noindent
 \textup{(ii)} The estimate \eqref{keytrilinests} also holds for $T=+\infty$ 
 and for any $1\le j\le 4$, we have 
 \begin{equation}
 \label{Iinftylimit}
  \big\| I_{T,j}(u_1,u_2,u_3) -  I_{+\infty,j}(u_1,u_2,u_3)   \big\|_{Y^s} \to 0 \text{ as } T\to +\infty%\,,\ 1\le j\le 4\,,
 \end{equation}
 and the limit
 \begin{equation}
 \label{Iinftylimit2}
 \lim_{t\to +\infty} S(-t)I_{+\infty,j}(u_1,u_2,u_3)(t)
 \end{equation}
  exists in $L^2(\R^2)$. 
 \end{proposition}
 \begin{proof} 
 \noindent
 \textup{(i)} 
Note that it suffices to prove the estimates for $s=0$ 
and by  Littlewood-Paley decomposition it suffices to prove the following
 \begin{equation}
 \label{Y0norm}
 \Bigg( \sum_{N_4} \Big\|  \sum_{N_1, N_2, N_3} P_{N_4}  \mathcal{D}_j(P_{N_1}u_1, P_{N_2}u_2, P_{N_3}u_3)  \Big\|_{U^2_S}^2 \Bigg)^{\frac12} 
 \lesssim  \|u_1\|_{Y^0}  \|u_2\|_{Y^0}  \|u_3\|_{Y^0} \,,
 \end{equation}
 where
 \begin{equation}
 \label{Duhamelint_j}
 \mathcal{D}_j(P_{N_1}u_1, P_{N_2}u_2, P_{N_3}u_3) :=
  \int_0^t  S(t-t') \mathcal{N}_j\big( P_{N_1}u_1, P_{N_2}u_2, P_{N_3}u_3\big)(t') \, dt' \,,
 \end{equation}
  with  $u_{i}:=\ind_{[0,T)} u_i$, $1\le i\le 4$. 
% , and where 
% $u_4\in C_c^{\infty}(\R^3)$ with $\|u_4\|_{V^2_S}=1$. 
 We treat the contribution of the worst parts of the nonlinearities $\mathcal{N}_j$'s, 
namely when there is a derivative and it falls on the largest frequency factor; 
the nonlinearity involving the Riesz transform $\mathcal{R}_x$ (which is bounded on $L^2(\R^2)$) 
does not pose a further difficulty.  

Thus, in what follows we focus on the nonlinear term 
  \begin{equation}
 \label{Itrilin}
 \mathcal{N}(P_{N_1}u_1, P_{N_2}u_2, P_{N_3}u_3) :=
  (P_{N_1}u_1) (P_{N_2}u_2) \dx(P_{N_3}u_3) 
 \end{equation}
 with $N_1\le N_2\le N_3$. 
Thus, by  Lemma \ref{lem:dualityS},\footnote{The estimates for $T=+\infty$ need to be derived separately as in this case we cannot apply 
Lemma~\ref{lem:dualityS}.} it remains to estimate 
 the following term
 \begin{equation}
 \label{tocontrol}
  \bigg( \sum_{N_4} \sup_{\|u_4\|_{V^2_S}=1}  \big|  I_{N_1,\ldots, N_4} \big|^2  \bigg)^{\frac12} \,,
 \end{equation}
 where 
 \begin{equation*}
 I_{N_1,\ldots, N_4}  :=  
   \sum_{N_1\le N_2\le N_3} N_3 
     \int_{\R^3} (P_{N_1}u_{1}) (P_{N_2} u_{2})  (P_{N_3} u_{3}) (P_{N_4} u_{4})  \, dx dy dt \,,
 \end{equation*}
 and where we assume without loss of generality that 
the Fourier transforms $\mathcal{F}_{xy}(u_i)$, $1\le i \le 4$,  
are non-negative. 
 
 We recall that due to the frequency hyperplane, 
 the largest two dyadic numbers are comparable and thus we distinguish the following cases. 
 
 \noindent
 \textbf{Case 1}: $N_1\le N_2 \ll N_3\sim N_4$. 
We apply the bilinear estimates \eqref{coro:bilinStric.1} and \eqref{coro:bilinStric.3} 
 %(see Corollary~\ref{coro:bilinStric}) 
 together with \eqref{eq:U2SV2Sembed}. 
 For fixed $N_4$, we have 
 \begin{align*}
%&\Big\|  \sum_{N_1\le N_2\le N_3} P_{N_4}  I(P_{N_1}u_1, P_{N_2}u_2, P_{N_3}u_3)  \Big\|_{U^2_S} \\
\big|  I_{N_1,\ldots, N_4} \big|
& \lesssim    \sum_{N_1\le N_2\ll N_3\sim N_4} 
       N_3 \| P_{N_1}u_1 P_{N_3}u_3\|_{L^2_{x,y,t}}  \| P_{N_2}u_2  P_{N_4}u_4\|_{L^2_{x,y,t}} \\
& \lesssim    \|u_1\|_{Y^0}  \|u_2\|_{Y^0} \sum_{N_1\le N_2\ll N_3\sim N_4} \frac{N_1^{\frac12} N_2^{\frac12-\varepsilon}}{N_4^{1-\varepsilon}} 
     \|P_{N_3} u_3\|_{U^2_S}  \|P_{N_4}  u_4\|_{V^2_S} \\
 & \lesssim    \|u_1\|_{Y^0}  \|u_2\|_{Y^0}   \sum_{ N_3\sim N_4}    \| P_{N_3}  u_3\|_{U^2_S}    \,,
 \end{align*}
 where in the last step we used  $\|P_{N_4} u_4\|_{V^2_S}\le 1$. 
 Then, by the Cauchy-Schwarz inequality we have
   \begin{align*}
\eqref{tocontrol} \lesssim  \|u_1\|_{Y^0}  \|u_2\|_{Y^0} \Bigg( \sum_{N_4}
   \bigg( \sum_{ N_3\sim N_4}    \| P_{N_3}  u_3\|_{U^2_S}  \bigg)^2\Bigg)^{\frac12}
   \lesssim  \|u_1\|_{Y^0}  \|u_2\|_{Y^0}  \|u_3\|_{Y^0} \,.
  \end{align*}
  
\noindent
 \textbf{Case 2}: $N_1, N_4 \ll N_2\sim N_3$. By arguing similarly to {Case 1}, for fixed $N_4$ we now have 
  \begin{align*}
\big|  I_{N_1,\ldots, N_4} \big| & \lesssim 
  \|u_1\|_{Y^0} 
    \sum_{N_1 \ll N_2\sim N_3} \frac{N_1^{\frac12} N_4^{\frac12-\varepsilon}}{N_2^{1-\varepsilon}} 
     \|P_{N_2} u_2\|_{U^2_S}  \|P_{N_3} u_3\|_{U^2_S}  \|P_{N_4}  u_4\|_{V^2_S}\\
     &\lesssim   \|u_1\|_{Y^0} 
    \sum_{N_2\sim N_3} \Big(\frac{ N_4}{N_2}\Big)^{\frac12-\varepsilon} 
     \|P_{N_2} u_2\|_{U^2_S}  \|P_{N_3} u_3\|_{U^2_S}\,.
\end{align*}
By Minkowski and Cauchy-Schwarz inequalities we then get 
\begin{align*}
\eqref{tocontrol} 
  &  \lesssim  \|u_1\|_{Y^0}  \sum_{N_2\sim N_3} 
    \bigg( \sum_{N_4\ll N_2}\Big(\frac{ N_4}{N_2}\Big)^{1-2\varepsilon}   \bigg)^{\frac12} 
   \|P_{N_2} u_2\|_{U^2_S}  \|P_{N_3} u_3\|_{U^2_S} 
   \lesssim  \|u_1\|_{Y^0}  \|u_2\|_{Y^0}  \|u_3\|_{Y^0} \,.
\end{align*}
 
  \noindent
 \textbf{Case 3}: $N_1\ll N_2 \sim N_3\sim N_4$. 
 We apply the bilinear estimate of Corollary~\ref{coro:bilinStric} and the {$L^4$-Strichartz} estimate \eqref{Strichartzlin.2} twice. Thus, arguing similarly to {Case 1}, we obtain
  \begin{align*}
%\textup{LHS of } \eqref{Y0norm} 
\eqref{tocontrol} 
\, &\lesssim  \Bigg( \sum_{N_4}
   \bigg( \sum_{N_1\ll N_2\sim N_3\sim N_4} 
   N_3 \| P_{N_1}u_1 P_{N_4}u_4\|_{L^2_{x,y,t}}  \|  P_{N_2}u_2\|_{L^4_{x,y,t}}  \|  P_{N_3}u_3\|_{L^4_{x,y,t}}  \bigg)^2\Bigg)^{\frac12} \\
 &\lesssim    \|u_1\|_{Y^0}  \|u_2\|_{Y^0}     \|u_3\|_{Y^0}      \,.
 \end{align*}
 
  \noindent
 \textbf{Case 4}: $N_4\ll N_1 \sim N_2\sim N_3$. 
 In this case the desired estimate follows as in Case~3 via an application of Minkowski's inequality 
 (similarly to the last step of Case~2). 
 
  \noindent
 \textbf{Case 5}: $1\ll N_1\sim N_2 \sim N_3\sim N_4$. 
 In this case we  use the $L^4$-Strichartz estimate \eqref{Strichartzlin.2} 
 together with \eqref{eq:U2SV2Sembed} %and the Cauchy-Schwarz inequality 
 in order to get
   \begin{align*}
 % \textup{LHS of } \eqref{Y0norm} \, 
\eqref{tocontrol}  &\lesssim \Bigg( \sum_{N_4}
   \bigg( \sum_{N_1\ll N_2\sim N_3\sim N_4}    N_3 \| P_{N_1}u_1\|_{L^4_{x,y,t}}  \| P_{N_2}u_2\|_{L^4_{x,y,t}}   \|  P_{N_3}u_3\|_{L^4_{x,y,t}}  \|  P_{N_4}u_4\|_{L^4_{x,y,t}} \bigg)^2\Bigg)^{\frac12}  \\
 &\lesssim    \|u_1\|_{Y^0}  \|u_2\|_{Y^0}     \|u_3\|_{Y^0}    \,.
 \end{align*}
 
 \noindent
 \textbf{Case 6}: $N_1\sim N_2 \sim N_3\sim N_4 \lesssim 1$. 
 We simply use the bilinear estimate \eqref{coro:bilinStric.5} twice together with \eqref{eq:U2SV2Sembed} 
 and we get \eqref{Y0norm}. 
 
 \medskip 
 
Part \textup{(ii)} follows analogously to the proof of \cite[Corollary~3.4]{HHK}. 
 %In order to prove the estimates  \eqref{keytrilinests} with $T=\infty$ we need a limiting argument. 
 
 \end{proof}

%\newpage
\begin{proof}[Proof of Theorem~\ref{thm:GWPsmalldata}] 
It uses a standard argument via contraction mapping principle applied to the simplified equation~\eqref{normD} 
after which we undo the the transformation \eqref{cov:normD}. 

Let $v_0\in L^2(\R^2)$ such that $ \|v_0\|_{L^2(\R^2)}<\delta$, with $\delta$ to be chosen later.  
First, we construct a solution $u(t)$ to \eqref{normD} for $t\in (0,\infty)$ with 
initial data $u_0(x,y) = e^{ia_2x} v_0(x, \sqrt{2} y)$, for some $a_2\in\R$ 
(see Lemma~\ref{lem:cov} and Remark~\ref{rmk:cov}). 
Thus, 
 we consider $\Phi(u)$ given by
\begin{equation}
\Phi(u)(t) = \ind_{[0,\infty)}(t) S(t) u_0 + \sum_{j=1}^4 c_j 
   \ind_{[0,\infty)}(t)   \int_0^t S(t-t') \mathcal{N}_j(u,u, u)(t') dt' \,.
\end{equation} 
Since $\ind_{[0,\infty)}(t)u_0$ is a $U^2$-atom, 
one easily checks that $\|\ind_{[0,\infty)}(t) S(t) u_0\|_{Y^0} \sim \|u_0\|_{L^2}$ 
and thus by Proposition~\ref{prop:trilinests}, we get 
$$\|\Phi(u)\|_{Y^0} \le C\delta + C\|u\|^3_{Y^0}\,,$$
for some $C>0$. 
By setting $r:=2C\delta$, we have that $\Phi$ maps $B_r:=\{u\in Y^0 : \|u\|_{Y^0} \le r\}$ into itself, 
provided that $8C^3\delta^2\le 1$. 
Similarly, by using telescoping sums, there exists 
$\widetilde{C}>0$ such that 
$$\|\Phi(u_1) - \Phi(u_2)\|_{Y^0} \le \widetilde{C} \|u_1-u_2\|_{Y^0} 
 \big(\|u_1\|_{Y^0}^2 + \|u_2\|_{Y^0}^2 \big) \,. $$
 By ensuring that $8C^2\widetilde{C} \delta^2<1$, 
  we have that $\Phi$ is a contraction on $B_r$, 
  and thus it has a unique fixed point in $B_r$ 
  which solves  \eqref{normD} for $t\in [0,+\infty)$. 
  %, or equivalently  \eqref{normD} has a unique solution in $Y^s$ 
  
  For the scattering claim, we use Proposition~\ref{prop:trilinests}~(ii) and take
  \begin{equation*}
  u_0^+ := u_0 + \sum_{j=1}^4 c_j \lim_{t\to +\infty} S(-t)I_{+\infty,j}(u,u,u)(t) \ \in L^2(\R^2) \,.
  \end{equation*}
  Since the embedding $Y^0\subset L^{\infty}(\R:L^2(\R^2))$ is continuous, we  get that 
  %$$\|u(t) - S(t)u_0^{+}\|_{L^2(\R^2)}\to 0\text{ as } t\to +\infty \,.$$ 
  for any $T>t$, %$\|u(t) - S(t)u_0^{+}\|_{L^2(\R^2)}$ is bounded by
% \begin{align*}
% \sum_{j=1}^4 |c_j|
%  \Big( \big\|I_{T,j}(t) -I_{+\infty,j}(t)\|_{L(\R^2)} 
%    + \Big\|S(t)\Big( S(-t) I_{+\infty,j}(t) - \lim_{\tau\to+\infty}  S(-\tau) I_{+\infty,j}(\tau)\Big)\Big\|_{L^2(\R^2)}
% \end{align*}
\begin{align*}
\|u(t) - S(t)u_0^{+}\|_{L^2(\R^2)} 
 \le 
 \sum_{j=1}^4 |c_j|
  \Big( & \big\|I_{T,j}(u,u,u) -I_{+\infty,j}(u,u,u)\|_{Y^0} \\
    & + \big\|S(-t) I_{+\infty,j}(u,u,u)(t) - \lim_{\tau\to+\infty}  S(-\tau) I_{+\infty,j}(u,u,u)(\tau)\big\|_{L^2(\R^2)} \Big)
 \end{align*}
 and by \eqref{Iinftylimit} and \eqref{Iinftylimit2}, 
 we then get $\|u(t) - S(t)u_0^{+}\|_{L^2(\R^2)} \to 0$ as $t\to +\infty$. 
% and equivalently that 
%$v(t)=\mathcal{T}^{-1}(u)(t)$ scatters to $S(t)v_0^+$ as $t\to +\infty$, where 
%$v_0^{+}(x,y)=e^{-a_2x} u_0^+(x,y)$. 

Second,  we construct the solution 
$u_{\textup{neg}}(t)$ to \eqref{normD} for $t\in (-\infty,0)$ and the scattering data $u_0^{-}$, 
by running the analogous argument
for the equation obtained from \eqref{normD} after the change of variable 
\begin{equation}
\label{timerev}
u(t,x,y) \mapsto \overline{u(-t,-x,-y)} =: \mathcal{I}(u)\,,
\end{equation} 
with initial data $\overline{u_0(-x,-y)}$, and then undoing the transformation \eqref{timerev}. 
Since the transformations \eqref{cov:normD} and \eqref{timerev} commute, 
the solution $v$ to the original Dysthe equation \eqref{D} is given by 
\begin{equation}
\label{defn:YYs}
 v= \mathcal{T}^{-1}(u +  u_{\textup{neg}}) 
 \in \Y  := 
 %  \mathcal{T}^{-1}\big(\ind_{[0,\infty)} Y^s + \ind_{(-\infty,0)}Y^s\big)\,.$$
 \big\{ \ind_{[0,\infty)}v_1 + \ind_{(-\infty,0)} v_2 : \mathcal{T}(v_1) \in Y^0 \text{ and } 
  \mathcal{T}(\mathcal{I}(v_2)) \in Y^0 \big\}\,,
\end{equation}
and the corresponding scattering data are 
$v_0^{\pm}(x,y):=e^{-ia_2x}u_0^{\pm}(x,\frac{y}{\sqrt{2}})$. 

\end{proof}

\medskip

In order to obtain the local well-posedness result for arbitrarily large initial data 
(i.e. Theorem~\ref{thm:LWP}), 
we prove the trilinear estimates in   $X^{s,b}$ spaces.

\begin{proposition}
Let $s>0$ and $\nu>0$ sufficiently small. 
For all $1\le j\le 4$, we have
\begin{equation}
\big\| \mathcal{N}_j(u_1,u_2,u_3) \big\|_{X^{s,-\frac12+2\nu}} \lesssim 
  \|u_1\|_{X^{s,\frac12+\nu}} \|u_2\|_{X^{s,\frac12+\nu}}  \|u_3\|_{X^{s,\frac12+\nu}} \,.
\end{equation}
\end{proposition}
\begin{proof}
As in the proof of Proposition~\ref{prop:trilinests}, we focus on the contribution of the worst parts of the 
nonlinearities and thus we prove the estimate for the nonlinear term 
\begin{equation*}
 \mathcal{N}(P_{N_1}u_1, P_{N_2}u_2, P_{N_3}u_3) :=
  (P_{N_1}u_1) (P_{N_2}u_2) \dx(P_{N_3}u_3) 
 \end{equation*}
 with $N_1\le N_2\le N_3$. 
By duality and  Littlewood-Paley decomposition it suffices to prove the following
\begin{equation}
\label{estpf:Prop4p2}
\Bigg(
\sum_{\substack{ N_1\le N_2\le N_3\\ N_4\\ L_1,L_2,L_3,L_4}} 
 \Gamma^2_{N_1,\ldots, N_4} \Lambda^2_{L_1,\ldots, L_4}
  \big| I_{N_1,\ldots, N_4}^{L_1,\ldots,L_4}(v_1,v_2,v_3,v_4) \big|^2
    \Bigg)^{\frac12} 
     \lesssim\, \prod_{i=1}^4 \|v_i\|_{L^2(\R^3)}\,,
\end{equation}
where
\begin{align}
 &\Gamma_{N_1,\ldots, N_4}:= 
   N_1^{-s} N_2^{-s} N_3^{1-s} N_4^s \,,\\
&\Lambda_{L_1,\ldots,L_4}:=
  L_1^{-\frac12-\nu} L_2^{-\frac12-\nu} L_3^{-\frac12-\nu} L_4^{-\frac12+2\nu} \,,\\
& I_{N_1,\ldots, N_4}^{L_1,\ldots,L_4}(v_1,v_2,v_3,v_4):=
  \int_{\R^3} (P_{N_1}Q_{L_1}v_1) (P_{N_2}Q_{L_2}v_2) (P_{N_3}Q_{L_3}v_3) (P_{N_4}Q_{L_4}v_4)
   \,dxdy dt
  \,,
\end{align}
and where we assume without loss of generality that 
the Fourier transforms $\mathcal{F}_{xy}(v_i)$, $1\le i \le 4$,  
are non-negative. 

We discuss the same cases as in the proof of Proposition~\ref{prop:trilinests}. 

 \noindent
 \textbf{Case 1}:  $N_1\le N_2 \ll N_3\sim N_4$. As in the proof of \cite[Proposition~4.1]{MoPi}, 
 we interpolate between the two bilinear Strichartz estimates 
 \eqref{BilinStrichartzI0} and \eqref{BilinStrichartzI1},
  and so we have
 \begin{equation} 
\label{BilinStrichartz-interp}
\begin{split}
&\|(P_{N_2}Q_{L_2}u_2)(P_{N_4}Q_{L_4}u_4)\|_{L^2(\R^3)}\\
&\qquad\quad  \lesssim \,
 \frac{N_2^{\frac{1+\theta}{2}}}{N_4^{1-\theta}} \min\{L_2, L_4\}^{\frac12}
   \max\{L_2 ,L_4\}^{\frac{1-\theta}{2}} 
    \|P_{N_2}Q_{L_2}u_2\|_{L^2(\R^3)} \|P_{N_4}Q_{L_4}u_4\|_{L^2(\R^3)} \,.
\end{split}
\end{equation}
We now assume that $L_2\le L_4$ as the case $L_4< L_2$ is easier. 
Then, by the Cauchy-Schwarz inequality,   \eqref{BilinStrichartz-interp}, and 
\eqref{BilinStrichartzI1} for $(P_{N_1}Q_{L_1}u_1)(P_{N_3}Q_{L_3}u_3)$, 
we get 
\begin{equation}
\label{eq:prop4p2case1}
\begin{split}
&\Gamma_{N_1,\ldots, N_4} \Lambda_{L_1,\ldots,L_4}
 \big| I_{N_1,\ldots, N_4}^{L_1,\ldots,L_4}(v_1,v_2,v_3,v_4)  \big|\\
&\qquad \quad
\lesssim  \,
\frac{N_1^{\frac12-s} N_2^{\frac{1+\theta}{2} -s } }{N_4^{1-\theta}} 
L_1^{-\nu} L_2^{-\nu} L_3^{-\nu} L_4^{2\nu-\frac{\theta}{2}} \prod_{i=1}^4 \|v_i\|_{L^2(\R^3)} \,.
\end{split}
\end{equation}
By choosing $\theta$ and $\nu$ such that 
$ \frac{1+\theta}{2} -s + \frac12 - s <1-\theta$ and $2\nu - \frac{\theta}{2}<0$, 
the estimate \eqref{estpf:Prop4p2} follows immediately. 

 \noindent
 \textbf{Case 2}:  $N_1, N_4 \ll N_2\sim N_3$. The estimate \eqref{estpf:Prop4p2} follows as in {Case~1} above (by interchanging $N_2$ and $N_4$ in \eqref{eq:prop4p2case1}) 
 via Minkowski's inequality for the summation in $N_4$ 
 (see also Case~2 in the proof of Proposition~\ref{prop:trilinests}).

For the remaining  cases, we follow essentially the same arguments as in the proof of Proposition~\ref{prop:trilinests}. 

  \noindent
 \textbf{Case 3}:  $N_1\ll N_2 \sim N_3\sim N_4$. 
 We apply the bilinear Strichartz estimate \eqref{BilinStrichartzI1} and 
 the $L^4$-Strichartz estimate \eqref{Strichartzlin.4interp} twice 
 (with $\theta=\frac{\nu}{2}$ and $\theta=3\nu$) and we get
 \begin{align*}
\Gamma_{N_1,\ldots, N_4} \Lambda_{L_1,\ldots,L_4}
 \big|I_{N_1,\ldots, N_4}^{L_1,\ldots,L_4}(v_1,v_2,v_3,v_4) \big| 
%&\qquad 
\lesssim N_1^{-s} 
%N_2^{-1-s} N_3^{\frac34-s} N_4^{-\frac14+s+\frac{9\nu}{1+6\nu}} 
N_2^{-s + \frac{3\nu}{2(1+\nu)} + \frac{9\nu}{1+6\nu}} 
% L_1^{-\nu}  L_2^{-\nu}  L_3^{-\frac{3\nu}{2}}  L_4^{-\nu}
\prod_{i=1}^4 L_i^{-\nu} \|v_i\|_{L^2(\R^3)} \,.
 \end{align*}
 By choosing $\nu>0$ such that $ \frac{3\nu}{2(1+\nu)}  + \frac{9\nu}{1+6\nu}<s$ 
 we ensure summability over  $N_2\sim N_3\sim N_4$ and thus 
 we get \eqref{estpf:Prop4p2}. 
 
  \noindent
 \textbf{Case 4}: $N_4\ll N_1 \sim N_2\sim N_3$. 
 We use the interpolated bilinear Strichartz estimate as in Case~1 
 (namely \eqref{BilinStrichartz-interp} with 
 ${N_2^{\frac{1+\theta}{2}}}/{N_4^{1-\theta}}$ 
 replaced by 
 ${N_4^{\frac{1+\theta}{2}}}/{N_2^{1-\theta}} $) 
 and 
 the $L^4$-Strichartz estimate \eqref{Strichartzlin.4interp}. 
 We have 
  \begin{align*}
\Gamma_{N_1,\ldots, N_4} \Lambda_{L_1,\ldots,L_4}
 \big|I_{N_1,\ldots, N_4}^{L_1,\ldots,L_4}(v_1,v_2,v_3,v_4) \big| 
\lesssim N_1^{\theta-3s} N_4^{\frac{\theta}{2}}  
L_1^{-\nu} L_2^{-\nu} L_3^{-\nu} L_4^{2\nu-\frac{\theta}{2}} \prod_{i=1}^4 \|v_i\|_{L^2(\R^3)}
  \end{align*}
  and we choose $\theta$ and $\nu$ such that 
 $0<4\nu <\theta<2s$.

  \noindent
 \textbf{Case 5}:  $1\ll N_1\sim N_2 \sim N_3\sim N_4$. 
  We use the $L^4$-Strichartz estimate \eqref{Strichartzlin.4interp} to obtain 
  \begin{align*}
\Gamma_{N_1,\ldots, N_4} \Lambda_{L_1,\ldots,L_4}
 \big| I_{N_1,\ldots, N_4}^{L_1,\ldots,L_4}(v_1,v_2,v_3,v_4) \big| 
\lesssim  N_1^{-2s+\frac{9\nu}{2(1+\nu)} + \frac{9\nu}{1+6\nu}} 
  \prod_{i=1}^4 L_i^{-\nu} \|v_i\|_{L^2(\R^3)} 
  \end{align*}
  and we choose $\nu>0$ such that $ \frac{9\nu}{2(1+\nu)}  + \frac{9\nu}{1+6\nu}<2s$.

  \noindent
 \textbf{Case 6}:  $N_1\sim N_2 \sim N_3\sim N_4 \lesssim 1$. 
 We simply use the bilinear estimate \eqref{BilinStrichartzI0} twice to obtain 
  \begin{align*}
\Gamma_{N_1,\ldots, N_4} \Lambda_{L_1,\ldots,L_4}
  \big| I_{N_1,\ldots, N_4}^{L_1,\ldots,L_4}(v_1,v_2,v_3,v_4) \big|
\lesssim L_1^{-\nu} L_2^{-\nu} L_3^{-\nu} L_4^{-\frac12+2\nu}
  \prod_{i=1}^4  \|v_i\|_{L^2(\R^3)} 
  \end{align*}
and we choose $0<\nu<\frac14$. 

\end{proof}

\begin{proof}[Proof of Theorem~\ref{thm:LWP}] 
It uses a standard argument via contraction mapping principle and Proposition~\ref{prop:basicXsb} 
applied to 
\begin{equation}
\Phi(u)(t) = \chi(t) S(t) u_0 + \sum_{j=1}^4 c_j 
   \chi(t)   \int_0^t \chi(t'/T) S(t-t') \mathcal{N}_j(u,u, u)(t') dt' 
\end{equation} 
in a ball of $X_T^{s,\frac12+\nu}$, by choosing $T=C \|u_0\|^{-\frac{2}{\nu}}_{H^s(\R^2)}$, 
$\nu>0$ sufficiently small, 
some  constant $C>0$, and some $\chi\in C_0^{\infty}(\R)$ with $\chi\equiv 1$ on $[0,1]$. 
Lastly, we note that the solution $v$ corresponding to the original equation lies in the space 
$\X_T^{s,b}$ defined analogous to $X^{s,b}_T$,  
using the non-homogeneous symbol $\omega$ (defined by \eqref{defn:omega}) 
instead of the homogeneous version $w$.

\end{proof}

\medskip

\section{Final comments}
\label{sect:finalcomm}
Some interesting issues remain open for the Cauchy problem of Dysthe type equations, for instance the possible finite time blow-up of large local solutions. 

Other questions concern the Dysthe equations as a water waves model and are summarized in \cite{La} Section 8.5.4, for instance the  comparison of the Dysthe equation with other models with improved dispersion.

More precisely, as recalled in Lannes  \cite{La}, Section 8.5.4, the Dysthe equation contains both higher nonlinear and dispersive terms neglected in the standard cubic NLS equation, while the full dispersion equation derived in \cite{La} Section 8.5.3 contains only higher order dispersive terms (at infinite order actually). A comparison between these models (and between a possible full-dispersion Dysthe equation)  would bring some insight to the relative importance of dispersive and nonlinear effects in situations where the standard NLS approximation compares poorly with the experiments (see for instance  \cite{MY}).  We plan to come back to those issues in a subsequent paper.

%Finally there are interesting higher order NLS type equations modeling various phenomena in nonlinear optics (see \cite{PRLzozulya, BSNKW}). that we will study in 
%a forthcoming  paper.
\bigskip

\begin{merci}
D.P. and R.M. were supported by a Trond Mohn Foundation grant. J.-C. S.  was partially  supported by the ANR project ANuI (ANR-17-CE40-0035-02). The authors would like to thank Karsten Trulsen for several interesting comments which clarified for us the derivation of the models and for relevant additional references and Felipe Linares for correcting several typos and pointing out that the linear Strichartz estimate  associated to the homogeneous Dysthe symbol was already proved in \cite{KPV}.
\end{merci}

\end{document}